\documentclass{amsart}
\usepackage{amssymb,amsthm,amsmath,epsfig,latexsym}
\usepackage{calc}
\usepackage{amscd,amssymb,subfigure,hyperref}
\usepackage[arrow,matrix,graph,frame,poly,arc,tips]{xy}
\usepackage{psfrag}

\begin{document}

\newcommand{\mmbox}[1]{\mbox{${#1}$}}
\newcommand{\affine}[1]{\mmbox{{\mathbb A}^{#1}}}
\newcommand{\Ann}[1]{\mmbox{{\rm Ann}({#1})}}
\newcommand{\caps}[3]{\mmbox{{#1}_{#2} \cap \ldots \cap {#1}_{#3}}}
\newcommand{\N}{{\mathbb N}}
\newcommand{\Z}{{\mathbb Z}}
\newcommand{\R}{{\mathbb R}}
\newcommand{\Q}{{\mathbb Q}}
\newcommand{\A}{{\mathcal A}}
\newcommand{\B}{{\mathcal B}}
\newcommand{\C}{{\mathbb C}}
\newcommand{\PP}{{\mathbb P}}
\newcommand{\kk}{{\mathbb K}}
\newcommand{\Tor}{\mathop{\rm Tor}\nolimits}
\newcommand{\Ext}{\mathop{\rm Ext}\nolimits}
\newcommand{\Hom}{\mathop{\rm Hom}\nolimits}
\newcommand{\im}{\mathop{\rm Im}\nolimits}
\newcommand{\Ass}{\mathop{\rm Ass}\nolimits}
\newcommand{\rank}{\mathop{\rm rank}\nolimits}
\newcommand{\codim}{\mathop{\rm codim}\nolimits}
\newcommand{\pdim}{\mathop{\rm pdim}\nolimits}
\newcommand{\supp}{\mathop{\rm supp}\nolimits}
\newcommand{\HF}{\mathrm{HF}}
\newcommand{\spn}{\mathrm{Span}_{\kk}}
\newcommand{\reg}{\mathrm{reg}}
\newcommand{\HP}{\mathrm{HP}}
\newcommand{\HS}{\mathrm{HS}}
\newcommand{\VVert}{\nu}
\newcommand{\In}{\mathrm{in}}
\newcommand{\coker}{\mathop{\rm coker}\nolimits}
\sloppy
\newtheorem{defn0}{Definition}[section]
\newtheorem{prop0}[defn0]{Proposition}
\newtheorem{conj0}[defn0]{Conjecture}
\newtheorem{thm0}[defn0]{Theorem}
\newtheorem{lem0}[defn0]{Lemma}
\newtheorem{corollary0}[defn0]{Corollary}
\newtheorem{example0}[defn0]{Example}
\newtheorem{remark0}[defn0]{Remark}

\newcommand{\p}{\mathbb{P}}
\newcommand{\V}{\mathbb{V}}
\newcommand{\T}{\mathbb{T}}
\newcommand{\n}{\txbf{n}}
\newcommand{\dra}{\dashrightarrow}
\newcommand{\txbf}{\textbf}
\newcommand{\sset}{\subseteq}
\newcommand{\fns}{\footnotesize}
\newcommand{\scs}{\scriptsize}
\newcommand{\mcC}{\mathcal{C}}
\newcommand{\mcW}{W_d}
\newcommand{\txc}{\textcolor}
\newcommand{\txcbl}{\txc{Blue}}
\newcommand{\mcd}{\mathcal{D}}
\newcommand{\mcV}{\mathcal{V}}
\newcommand{\mcA}{\mathcal{A}}
\newcommand{\sss}{\scriptscriptstyle}
\newcommand{\bn}{\txbf{n}}

\newenvironment{defn}{\begin{defn0}}{\end{defn0}}
\newenvironment{prop}{\begin{prop0}}{\end{prop0}}
\newenvironment{conj}{\begin{conj0}}{\end{conj0}}
\newenvironment{thm}{\begin{thm0}}{\end{thm0}}
\newenvironment{lem}{\begin{lem0}}{\end{lem0}}
\newenvironment{cor}{\begin{corollary0}}{\end{corollary0}}
\newenvironment{exm}{\begin{example0}\rm}{\end{example0}}
\newenvironment{rmk}{\begin{remark0}}{\end{remark0}}

\newcommand{\msp}{\renewcommand{\arraystretch}{.5}}
\newcommand{\rsp}{\renewcommand{\arraystretch}{1}}

\newenvironment{lmatrix}{\renewcommand{\arraystretch}{.5}\small
 \begin{pmatrix}} {\end{pmatrix}\renewcommand{\arraystretch}{1}}
\newenvironment{llmatrix}{\renewcommand{\arraystretch}{.5}\scriptsize
 \begin{pmatrix}} {\end{pmatrix}\renewcommand{\arraystretch}{1}}
\newenvironment{larray}{\renewcommand{\arraystretch}{.5}\begin{array}}
 {\end{array}\renewcommand{\arraystretch}{1}}

\def \a{{\mathrel{\smash-}}{\mathrel{\mkern-8mu}}
{\mathrel{\smash-}}{\mathrel{\mkern-8mu}} {\mathrel{\smash-}}{\mathrel{\mkern-8mu}}}

\title[Geometry of Wachspress surfaces]%
{Geometry of Wachspress surfaces}
\author[Corey Irving]{Corey Irving}
\thanks{Irving supported by Texas Advanced Research Program 010366-0054-2007}
\address{Mathematics Department, Santa Clara University, Santa Clara CA 95053}
\email{\href{mailto:cfirving@scu.edu}{cfirving@scu.edu}}
\urladdr{\href{http://www.math.scu.edu/~cirving/}%
{http://www.math.scu.edu/\~{}cirving}}

\author[Hal Schenck]{Hal Schenck}
\thanks{Schenck supported by NSF 1068754, NSA H98230-11-1-0170}
\address{Mathematics Department,
University of Illinois Urbana-Champaign, Urbana, IL 61801}
\email{\href{mailto:schenck@math.uiuc.edu}{schenck@math.uiuc.edu}}
\urladdr{\href{http://www.math.uiuc.edu/~schenck/}%
{http://www.math.uiuc.edu/\~{}schenck}}


\begin{abstract}
\noindent Let $P_d$ be a convex polygon with $d$ vertices. 
The associated Wachspress surface $W_d$ is a fundamental object 
in approximation theory, defined as the image of the rational map 
\[
\mathbb{P}^2 \stackrel{w_d}{\longrightarrow} \mathbb{P}^{d-1},
\]
determined by the Wachspress barycentric coordinates for $P_d$. 
We show $w_d$ is a regular map on a blowup 
$X_d$ of $\mathbb{P}^2$ and if $d>4$ is given by a very ample divisor 
on $X_d$, so has a smooth image $W_d$. We determine
generators for the ideal of $W_d$, and prove that in
graded lex order, the initial ideal of $I_{W_d}$ is given by a 
Stanley-Reisner ideal. As a consequence, we show 
that the associated surface is arithmetically Cohen-Macaulay, of 
Castelnuovo-Mumford regularity two, and determine all the graded
betti numbers of $I_{W_d}$. 
\end{abstract}
\maketitle
\vskip -.3in
\section{Introduction}\label{sec:one}
Introduced by M\"{o}bius \cite{mob} in 1827, barycentric coordinates for
triangles appear in a host of applications. Recent work in approximation
theory has shown that it is also useful to define barycentric
coordinates for a convex polygon $P_d$ with $d \ge 4$ vertices (a $d$-gon). 
The idea is as follows: 
to deform a planar shape, first place the shape inside a control polygon.
Then move the vertices of the control polygon, and use barycentric coordinates 
to extend this motion to the entire shape. 

For a $d$-gon with $d \ge 4$, barycentric coordinates were defined by 
Wachspress \cite{wachs} in his work on finite elements; these coordinates 
are rational functions depending on the vertices $\nu(P_d)$ of $P_d$. In 
\cite{WarrUniq}, Warren shows that Wachspress' coordinates are 
the unique rational barycentric coordinates of minimal degree. 
The Wachspress coordinates define a rational map $w_d$ on $\PP^2$, whose value
at a point $p \in P_d$ is the $d$-tuple of barycentric coordinates
of $p$. The closure of the image of $w_d$ is the Wachspress surface 
$W_d$, first defined and studied by Garc{\i}a--Puente and 
Sottile \cite{gp} in their work on linear precision.

In Definition~\ref{theLinearforms} we fix linear forms ${\ell}_i$ which 
are positive inside $P_d$ and vanish on an edge. 
Let $A = {\ell}_1 \cdots {\ell}_d$, $Z$ be the ${d \choose 2}$ singular 
points of $\V(A)$ and $Y = Z \setminus \VVert(P_d)$. We call $Y$ the \emph{external vertices}
of $P_d$, and show that $w_d$ has basepoints only at $Y$. Let $X_d$ be the blowup of $\mathbb{P}^2$ at $Y$. 
In \S 2, we prove that $W_d$ is the image of $X_d$, embedded 
by a certain divisor $D_{d-2}$ on $X_d$. The global sections of $D_{d-2}$ 
have a simple interpretation in terms of the edges $\V({\ell}_i)$ of $P_d$: 
we prove that
\[
H^0(\mathcal{O}_{X_d}(D_{d-2})) \mbox{ has basis } \{ {\ell}_3 \cdots {\ell}_d, {\ell}_1 {\ell}_4 \cdots {\ell}_d, \ldots,  {\ell}_2 \cdots {\ell}_{d-1} \}.
\]
We show that $D_{d-2}$ is very ample if $d>4$, hence $W_d \subseteq \PP^{d-1}$ is a smooth surface.
\subsection{Statement of main results}
For a $d$-gon $P_d$ with $d \ge 4$
\begin{enumerate}
\item We give explicit generators for $I_{W_d} \subseteq S = \kk[x_1,\ldots, x_d]$.
\item We determine $\In_{\prec}(I_{W_d})$, where $\prec$ is graded lex order.
\item We prove $\In_{\prec}(I_{W_d})$ is the Stanley-Reisner ideal of a 
graph $\Gamma$.
\item We prove that $S/I_{W_d}$ is Cohen-Macaulay, and $\reg(S/I_{W_d})=2$.
\item We determine the graded betti numbers of $S/I_{W_d}$.
\end{enumerate}
In \S 1.2 we give some quick background on geometric modeling, 
and in \S 1.3 we do the same for algebraic geometry (in particular, 
we define all the terms above). Our strategy runs as follows. 
In \S 2, we study $I_{W_d}$ by blowing up $\mathbb{P}^2$ at the external vertices. Define a divisor 
\[
D_{d-2} = (d-2)E_0 - \sum\limits_{p \in Y}E_p
\]
on $X_d$, where $E_0$ is the pullback of a line and $E_p$ is the 
exceptional fiber over $p$. We show that $D_{d-2}$ is
very ample, and that $I_{W_d}$ is the ideal of the image of 
\[
X_d \longrightarrow \mathbb{P}(H^0(D_{d-2})).
\]
Riemann-Roch then yields the Hilbert polynomial of $S/I_{W_d}$.

In \S 3 and \S 4 we find distinguished sets of quadrics and cubics 
vanishing on $W_d$, and use them to generate a subideal $I(d) \subseteq I_{W_d}$. 
In \S 5 we tie everything together, showing that in graded lex
order, $I_\Gamma(d) \subseteq \In_{\prec}I(d)$, where $I_\Gamma(d)$ is the 
Stanley-Reisner ideal of a certain graph. Using results on
flat deformations and an analysis of associated primes, we prove
\[
I_\Gamma(d) = \In_{\prec}(I(d)).
\]
The description in terms of the Stanley-Reisner ring yields the Hilbert series
for $S/I_\Gamma(d)$. We prove that $S/I_\Gamma(d)$ is Cohen-Macaulay and 
has Castelnuovo-Mumford regularity two, and 
it follows from uppersemicontinuity that the same is true for $S/I(d)$.
The differentials on the quadratic generators of $I_\Gamma(d)$ turn out to be 
easy to describe, and combining this with the regularity bound and 
knowledge of the Hilbert series yields the graded betti numbers for 
$\In_{\prec}(I(d))$. 

Finally, we show that $I(d)$ has no linear syzygies on its 
quadratic generators, which allows us to prune the resolution of $\In_{\prec}(I(d))$ to obtain the graded betti numbers of $I(d)$. Comparing Hilbert polynomials
shows that up to saturation
\[
S/I(d) = S/I_{W_d}.
\]
Since $I_{W_d}$ is prime, it is saturated, and a short exact sequence argument shows that $S/I(d)$ is also saturated,
concluding the proof.

\subsection{Geometric Modeling Background}
Let $P_d$ be a $d$-gon with vertices $v_1,\dots,v_d$ and indices taken
modulo $d$.
\begin{defn} 
\label{bcdefinition}
Functions $\{\beta_i:P_d\to \R \mid 1\leq i\leq d\}$ are \emph{barycentric
coordinates} if for all $p \in P_d$:
\begin{center}
\begin{tabular}{lcr}
$1.\ \ \beta_i(p)\geq 0$\ \ \ \ \ \ \ \ \ &
$2.\ \ \displaystyle{p=\sum_{i=1}^d\beta_i(p)v_i}$ & \ \ \ \ \ \ \
\ \ \ 
$3.\ \ \displaystyle{\sum_{i=1}^d\beta_i(p)=1}.$
\end{tabular}
\end{center}
\end{defn}
Wachspress coordinates have a geometric description in terms of areas of
subtriangles of the polygon. Let $A(a,b,c)$ denote the area of
the triangle with vertices $a, b, $ and $c$.  For $1\leq j\leq d$ set
$\alpha_j:=A(v_{j-1},v_j,v_{j+1})$
and $A_j:=A(p,v_j,v_{j+1})$.
\begin{defn} \label{wachsDefinition}
For $1\leq i\leq d$, the functions 
$$\beta_i=\frac{b_i}{\sum_{j=1}^d b_j}, \mbox{ where }\displaystyle{b_i=\alpha_i\prod_{j\neq i-1,i}A_j}$$ are Wachspress barycentric coordinates for the $d$-gon $P_d$, see Figure \ref{wachspicture}. 
\end{defn}
\vspace{10pt}
\begin{figure}[htb]
\centering
\psfrag{vi}{\Large$v_i$}
\psfrag{v_j}{\Large$v_j$}
\psfrag{v_j+1}{\Large$v_{j+1}$}
\psfrag{x}{\Large$p$}
\psfrag{A_j}{\Large$A_j$}
\psfrag{ai}{\Large$\alpha_i$}
\includegraphics[scale=.30]{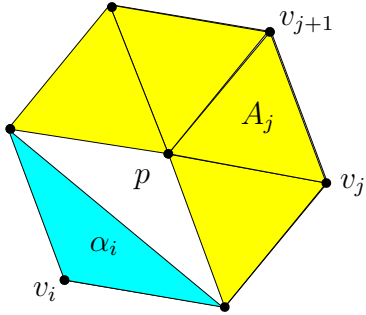}
 \caption{Wachspress coordinates for a polygon}
\label{wachspicture}
\end{figure}
We embed $P_d$ in the plane $z=1 \subseteq \R^3$,  
and form the cone with ${\bf 0} \in \R^3$. Explicitly, to each vertex 
$v_i \in \nu(P_d)$ we associate the ray ${\bf v}_i := (v_i,1) \in \R^3$.
Let ${\bf P}_d$ denote the cone generated by the rays ${\bf v}_i$, and 
$\nu({\bf P}_d):= \{{\bf v}_i | v_i \in \nu(P_d)\}$. The cone over the edge
$[v_i,v_{i+1}]$ corresponds to a facet of ${\bf P}_d$, with normal 
vector $\bn_i:={\bf v}_i\times {\bf v}_{i+1}$. We redefine 
$\alpha_j$ and $A_j$ to be the determinants $|{\bf v}_{j-1} {\bf v}_j {\bf v}_{j+1}|$ and $|{\bf v}_j {\bf v}_{j+1} {\bf p}|$, where ${\bf p}=(x,y,z)$.
This scales the $b_i$ by a factor of $2$, so leaves the $\beta_i$ unchanged,
save for homogenizing the $A_j$ with respect to $z$, and allows us to define 
Wachspress coordinates for non-convex polygons, although Property 1 of barycentric coordinates fails when $P_d$ is non-convex. 
\begin{defn}\label{theLinearforms}
${\ell}_j:=A_j=\bn_j\cdot {\bf p}=|{\bf v}_j {\bf v}_{j+1} {\bf p}|.$
\end{defn}
The ${\ell}_j$ are homogeneous linear forms in $(x,y,z)$, and vanish 
on the cone over the edge $[v_j,v_{j+1}]$. We use Theorem~\ref{independent} 
below, but Warren's proof does not require convexity. Our results 
hold over an arbitrary field $\kk$, as long as no three of the 
lines $\V(\ell_i) \subseteq \mathbb{P}^2$ meet at a point. 
For Condition 1 of Definition~\ref{bcdefinition} to make 
sense, $\kk$ should be an ordered field.
\begin{defn}
The \emph{dual cone} to ${\bf P}_d$ is the cone spanned by the normals $\bn_1,\dots,
\bn_d$ and is denoted ${\bf P}_d^*$.
\end{defn}
Triangulating $P_d$ yields a triangulation of ${\bf P}_d$, and the volume of the parallelepiped $S$ spanned by vertices $\{{\bf v}_i,{\bf v}_j,{\bf v}_k, {\bf 0}\}$ is $a_S=|{\bf v}_i {\bf v}_j {\bf v}_k|$.
\begin{defn}
Let C be a cone defined by a polygon $P_d$ and $T(C)$ a triangulation of $C$ obtained from a triangulation of $P_d$ as above. The \emph{adjoint} of $C$ is 
$$\mcA_{T(C)}({\bf p})=\sum_{S\in T(C)}a_S\prod_{{\bf v}\in \nu({\bf P}_d) \setminus \nu(S)}({\bf v}\cdot {\bf p}) \in \kk[x,y,z]_{d-3}.$$
\end{defn}
\begin{thm}$($Warren \cite{WarrWach}$)$\label{independent} $\mcA_{T(C)}({\bf p})$ is independent of the triangulation $T(C)$.
\end{thm}
\subsection{Algebraic Geometry Background}
Next, we review some background in algebraic geometry, referring
to \cite{eis1}, \cite{hart}, \cite{sch} for more detail. Homogenizing the 
numerators of Wachspress coordinates yields our main object of study:
\begin{defn} The Wachspress map defined by a polygon $P_d$ is the rational map 
$\PP^2\stackrel{w_d}{\dashrightarrow}\PP^{d-1}$, given on the open set 
$U_{z \ne 0} \subseteq \PP^2$ by $(x,y) \mapsto (b_1(x,y), \ldots, b_d(x,y))$. The Wachspress variety $W_d$ is the closure of the image of $w_d$. 
\end{defn} 
The polynomial ring $S= \kk[x_1,\ldots,x_d]$ 
is a graded ring: it has a direct sum decomposition
into homogeneous pieces. A finitely generated graded $S$-module $N$ 
admits a similar decomposition; if $s \in S_p$ and $n \in N_q$ then
$s\cdot n \in N_{p+q}$. In particular, each $N_q$ is a $S_0 = \kk$-vector space.
\begin{defn}\label{SmodN}
For a finitely generated graded $S$-module $N$, the Hilbert series
$HS(N,t) = \sum \dim_{\kk}N_q t^q$.
\end{defn}
\begin{defn}A free resolution for an $S$-module $N$ is an exact sequence 
\[
\mathbb{F} : \cdots \rightarrow F_i \stackrel{d_i}{\rightarrow} F_{i-1}\rightarrow \cdots \rightarrow F_0  \rightarrow N \rightarrow 0,
\]
where the $F_i$ are free $S$-modules. 
\end{defn}
If $N$ is graded, then the $F_i$ are also graded, so letting 
$S(-m)$ denote a rank one free module generated in degree $m$, we may 
write $F_i = \oplus_j S(-j)^{a_{i,j}}$. By the Hilbert syzygy theorem \cite{eis1} a finitely generated, graded $S$-module $N$ has 
a free resolution of length at most $d$, with all the $F_i$ of finite 
rank. 
\begin{defn} For a finitely generated graded $S$-module 
$N$, a free resolution is minimal if for each $i$,
$\im(d_i) \subseteq \mathfrak{m}F_{i-1}$, where 
$\mathfrak{m}=\langle x_1,\ldots,x_d\rangle$. 
The graded betti numbers of $N$ are the $a_{i,j}$ which appear 
in a minimal free resolution, and the Castelnuovo-Mumford 
regularity of $N$ is $\max_{i,j} \{a_{i,j}-i\}$. 
\end{defn}
While the differentials which appear in a minimal free resolution 
of $N$ are not unique, the ranks and degrees of the free modules which 
appear are unique. The graded betti numbers are displayed in 
a {\em betti table}. Reading this table right and down, starting 
at $(0,0)$, the entry $b_{ij} : = a_{i,i+j}$, and the regularity of $N$ is the 
index of the bottommost nonzero row in the betti table for $N$. 
\begin{exm}\label{braidex1}
In Examples 2.9 and 3.11 of \cite{gp} it is shown that $I_{W_6}$ is
generated by three quadrics and one cubic. The variety 
$\mathbb{V}({\ell}_1 \cdots {\ell}_6)$ of the edges of $P_6$ has 
${6 \choose 2}=15$ singular points, of which six 
are vertices of $P_6$, and $S/I_{W_6}$ has betti table
\begin{small}
$$
\vbox{\offinterlineskip 
\halign{\strut\hfil# \ \vrule\quad&# \ &# \ &# \ &# \ &# \ &# \ &# \ &# \ &# \ &#   \cr
total&1&4&6&3\cr \noalign {\hrule} 0&1 &--&--& -- \cr 1&--&3 &--& --
\cr 2&--&1 &6 &3 \cr \noalign{\bigskip} \noalign{\smallskip} }}
$$
\end{small}
For example, $b_{1,2} = a_{1,3} = 1$ reflects that $I_{W_6}$ 
has a cubic generator, and $\reg(S/I_{W_6})=2$. The 
Hilbert series can be read off the betti table:
\[
\HS(S/I_{W_6},t) = \frac{1 -3t^2-t^3+6t^4-3t^5}{(1-t)^6} = \frac{1 + 3t+3t^2}{(1-t)^3}.
\]
Theorem~\ref{BettiTable} gives a complete description of the betti table of $S/I_{W_d}$.
\end{exm}
\section{$H^0(D_{d-2})$ and the Wachspress surface}\label{sec:two}
\subsection{Background on blowups of $\PP^2$}
Fix points  $p_1, \ldots p_k \in \PP^2$, and let 
\begin{equation}\label{BlowP2}
X \stackrel{\pi}{\longrightarrow}\PP^2
\end{equation} 
be the blow up of ${\PP}^2$ at these points. Then $Pic(X)$ is
generated by the exceptional curves $E_i$ over the points $p_i$, 
and the proper transform $E_0$ of a line in ${\PP}^2$. 
A classical geometric problem asks for a relationship between 
numerical properties of a divisor $D_m = mE_0-\sum a_iE_i$ on $X$, 
and the geometry of $$X \stackrel{\phi}{\longrightarrow} \PP(H^0(D_m)^\vee).$$
First, some basics. Let $m$ and $a_i$ be non-negative, 
let $I_{p_i}$ denote the ideal of a point $p_i$, and define 
\begin{equation}\label{idealJ}
J = \bigcap\limits_{i=1}^k I_{p_i}^{a_i} \subseteq \kk[x,y,z]=R.
\end{equation}
Then $H^0(D_m)$ is isomorphic to the $m^{th}$ graded piece $J_m$ 
of $J$ (see \cite{h}). In \cite{DG}, Davis and 
Geramita show that if 
$\gamma(J)$ denotes the smallest degree $t$ such that $J_t$
defines $J$ scheme theoretically, then $D_m$ is very ample 
if $m > \gamma(J)$, and if $m=\gamma(J)$, then $D_m$ is very ample 
iff $J$ does not contain $m$ collinear points, counted with multiplicity. 
Note that $\gamma(J) \le reg(J)$. 

\subsection{Wachspress surfaces}
For a polygon $P_d$, fix defining linear forms ${\ell}_i$ as in 
Definition~\ref{theLinearforms} and let $A: ={\ell}_1\cdots {\ell}_d$; the 
edges of $P_d$ are defined by the $\V({\ell}_i)$. 
Let $Z$ denote the ${d \choose 2}$ 
singular points of $\V(A)$ and $Y = Z \setminus \VVert(P_d)$. Finally, $X_d$ will be the blowup of $\mathbb{P}^2$ at $Y$. 
We study the divisor
\[
D_{d-2} = (d-2)E_0 - \sum\limits_{p \in Y}E_p
\]
on $X_d$. First, some preliminaries.
\begin{defn}
Let $L$ be the ideal in $R=\kk[x,y,z]$ given by
\[
L  =  \langle {\ell}_3 \cdots {\ell}_d, {\ell}_1 {\ell}_4 \cdots {\ell}_d, \ldots,  {\ell}_2 \cdots {\ell}_{d-1} \rangle =  \langle A/{\ell}_1 {\ell}_2, A/{\ell}_2 {\ell}_3, \ldots, A/{\ell}_d {\ell}_1 \rangle,
\]
where $A = \prod_{i=1}^d {\ell}_i$.
\end{defn}
\noindent For any variety $V$, we use $I_V$ to denote the ideal of polynomials vanishing on $V$.  
\begin{lem}\label{schemeT}
The ideals $L$ and $I_Y$ are equal up to saturation at $\langle x,y,z \rangle$.
\end{lem}
\begin{proof}
Being equal up to saturation at $\langle x,y,z \rangle$ means that the
localizations at any associated prime except $\langle x,y,z \rangle$ are equal.
The ideal $I_p$ of a point $p$ is a prime ideal. Recall that the
localization of a ring $T$ at a prime ideal $\mathfrak{p}$ is a new ring 
$T_{\mathfrak p}$ whose elements are of the form $f/g$, with $f,g \in T$ and 
$g \not\in \mathfrak{p}$. 
Localize $R$ at $I_p$, where $p \in Y$. Then in $R_{I_p}$, ${\ell}_i$ is a unit if $p \not \in \V({\ell}_i)$. Without loss of generality, suppose forms ${\ell}_1$ and ${\ell}_2$ 
vanish on $p$ (note that all points of $Y$ are intersections of exactly two lines), and the remaining forms do not. Thus, $L_{I_p} = \langle {\ell}_1, {\ell}_2 \rangle = (I_Y)_{I_p}$.
\end{proof}
\noindent The ideal $L$ is not saturated.
\begin{lem}\label{gensL}
$I_Y$ is generated by one form $F$ of degree $d-3$ and $d-3$ 
forms of degree $d-2$. Hence a basis for $L_{d-2}$ consists of $F \cdot x, F \cdot y, F \cdot z$ and the $d-3$ forms. 
\end{lem}
\begin{proof}
First, note that $I_Y$ cannot contain any form of degree $d-4$, since
$Y$ contains $d$ sets of $d-3$ collinear points. So the smallest degree of
a minimal generator for $I_Y$ is $d-3$. Since $Y$ consists of
${d-1 \choose 2}-1$ distinct points, and the space of forms of degree
$d-3$ has dimension ${d-1 \choose 2}$, there is at least one form $F$
of degree $d-3$ in $I_Y$. We claim that it is unique. To see this, first
note that no ${\ell}_i$ can divide $F$: by symmetry if one ${\ell}_i$ divides 
$F$, they all must, which is impossible for degree reasons.
Now suppose $G$ is a second form of degree $d-3$ 
in $I_Y$. Let $p \in \nu(P_d)$ and $\V({\ell}_i)$ be a line corresponding to 
an edge containing $p$. $F(p)$ must be nonzero, since if not $\V(F)$ would
contain $d-2$ collinear points of $\V({\ell}_i)$, forcing $\V(F)$ to contain
$\V({\ell}_i)$, a contradiction. This also holds for $G$. But in this case,
$F(p)G - G(p)F$ is a polynomial of degree $d-3$ vanishing at $d-2$ collinear
points, again a contradiction. So $F$ is unique (up to scaling), which 
shows that the Hilbert function satisfies
\[
\HF(R/L, d-3) = |Y|,
\]
so $\HF(R/L, t) = |Y|$ for all $t \ge d-3$ (see \cite{sch}). As the polynomials
$A/{\ell}_i{\ell}_{i+1}$ are linearly independent and there are the 
correct number, $L_{d-2}$ must be the degree $d-2$ component of $I_Y$.
\end{proof}
\begin{thm}\label{regL}
The minimal free resolution of $R/L$ is
\begin{small}
\[
0 \longrightarrow R(-d) \xrightarrow{d_3} R(-d+1)^d \xrightarrow{d_2}
R(-d+2)^d
\xrightarrow{\left[ \!\begin{array}{cccc}
\frac{A}{{\ell}_1{\ell}_2}& \frac{A}{{\ell}_2{\ell}_3} &\cdots & \frac{A}{{\ell}_d{\ell}_1}
\end{array}\! \right]}
 R \longrightarrow R/L  \longrightarrow 0, 
\]
\end{small}
\begin{small}
\[
\mbox{ where }d_2 = {\left[ \!
\begin{array}{cccccccc}
{\ell}_1  & 0          & \cdots & \cdots & 0          & 0    &     m_1   \\
-{\ell}_3 & {\ell}_2   & 0      & \cdots  & \vdots  & \vdots      & m_2       \\
0         &  -{\ell}_4 & \ddots & \ddots  & \vdots & \vdots       & \vdots       \\ 
\vdots    & 0          & \ddots & \ddots        &  {\ell}_{d-2} &0     & \vdots    \\ 
\vdots    & \vdots     &  \ddots      &  \ddots & -{\ell}_d &{\ell}_{d-1}& \vdots   \\ 
0         & \cdots     & \cdots &  0      & 0 &-{\ell}_1        & m_d
\end{array}\! \right]}
\]
\end{small}
and the $m_i$ are linear forms.
\end{thm}
\begin{proof}
By Lemma~\ref{gensL} the generators of $I_Y$ are known. Since $I_Y$ is
saturated, the Hilbert-Burch theorem implies that the free resolution 
of $R/I_Y$ has the form
\[
0 \longrightarrow R(-d+1)^{d-3} \longrightarrow R(-d+3)\oplus R(-d+2)^{d-3}\longrightarrow R \longrightarrow R/I_Y \longrightarrow 0.
\]
Writing $I_Y$ as $\langle f_1,\ldots,f_{d-3},F \rangle$ and $L$ as 
 $\langle f_1,\ldots,f_{d-3},xF,yF,zF \rangle$, the task is to understand
the syzygies on $L$ given the description above of the syzygies on 
$I_Y$. From the Hilbert-Burch resolution, any minimal syzygy on 
$I_Y$ is of the form 
\[
\sum g_if_i +qF =0,
\]
where $g_i$ are linear and $q$ is a quadric (or zero). Since 
\[
qF = g_1xF+g_2yF+g_3zF \mbox{ with }g_i\mbox{ linear,}
\]
all $d-3$ syzygies on $I_Y$ lift to give linear syzygies on $L$. Furthermore,
we obtain three linear syzygies on $\{xF,yF,zF\}$ from the three Koszul syzygies on 
$\{x,y,z\}$. It is clear from the construction that these $d$ linear 
syzygies are linearly independent. Since $\HF(R/L, d-1) = |Y|$, this means
we have determined all the linear first syzygies. Furthermore, the three
Koszul first syzygies on $\{xF,yF,zF\}$ generate a linear second syzygy,
so the complex given above is a subcomplex of the minimal free resolution.
A check shows that the Buchsbaum-Eisenbud criterion \cite{be} holds,
so the complex above is actually exact, hence a free resolution. The 
differential $d_2$ above involves the canonical generators $A/{\ell}_i{\ell}_{i+1}$, rather than a set involving $\{xF,yF,zF\}$. Since the $d-1$ 
linear syzygies appearing in the first $d-1$ columns of $d_2$ are
linearly independent, they agree up to a change of basis; the last column 
of $d_2$ is a vector of linear forms determined by the change of basis.
\end{proof}
\vskip .1in
\begin{thm}~\label{sectionsD} 
$H^0(D_{d-2}) \!\simeq\! Span_{\kk}\{\frac{A}{{\ell}_1{\ell}_2}, \frac{A}{{\ell}_2{\ell}_3},\ldots \}, H^1(D_{d-2})\!=\!0\!=\!H^2(D_{d-2}).$
\end{thm}
\begin{proof}
The remark following Equation~\ref{idealJ} shows 
that $H^0(D_{d-2}) \simeq L_{d-2}$. Since $K = -3E_0+\!\sum\limits_{p \in Y} E_p$ (see \cite{hart}), by Serre duality 
\[
H^2(D_{d-2})\simeq H^0((-d-1)E_0+\!\sum\limits_{p \in Y} E_p),
\]
which is clearly zero. Using that $X_d$ is rational, it follows 
from Riemann-Roch that 
$$h^0(D_{d-2})-h^1(D_{d-2})=\frac{D_{d-2}^2-D_{d-2} \cdot K}{2}+ 1.$$
The intersection pairing on $X_d$ is given by 
$E_i^2  =  1$ if $i=0$, and $-1$ if $i \ne 0$, and 
\[
E_i \cdot E_j  =  0 \mbox{ if } i \ne j.
\]
Thus,
\begin{equation}\label{rr1}
\begin{aligned}
D_{d-2}^2 &= (d-2)^2 - |Y| & \mbox{ and } & -D_{d-2}K    & = 3(d-2) -|Y|,
\end{aligned}
\end{equation}
yielding
\begin{equation}\label{rr2}
\begin{aligned}
h^0(D_{d-2})-h^1(D_{d-2}) &= \frac{d^2-d-2-2|Y|}{2}+ 1&= d.
\end{aligned}
\end{equation}
Thus $h^0(D_{d-2})-h^1(D_{d-2}) = d$. Now apply the 
remark following Equation~\ref{idealJ}.
\end{proof}
\begin{cor}\label{isAmp}
If $d\!>\!4$, $D_{d-2}$ is very ample, so the image of $X_d$ in $\mathbb{P}^{d-1}$ is
smooth. 
\end{cor}
\begin{proof}
By Theorem~\ref{regL}, the ideal $L$ is $d-2$ regular. Furthermore, 
the set $Y$ contains $d$ sets of $d-3$ collinear points, but no set of $d-2$ collinear points if $d>4$. The result follows from the Davis-Geramita 
criterion. 
\end{proof}
\begin{thm} $W_4 \simeq \mathbb{P}^{1}\times \mathbb{P}^{1}$, and $X_4 \rightarrow W_4$ is an isomorphism away from the $(-1)$ curve $E_0-E_1-E_2$, which is contracted to a smooth point.
\end{thm}
\begin{proof} The surface $X_4$ is $\mathbb{P}^{2}$ blown up at two points, which is toric, and isomorphic to $\mathbb{P}^{1}\times \mathbb{P}^{1}$ blown up at a point. By Proposition 6.12 of \cite{cls}, $D_2$ is basepoint free. Since $D_2^2 =2$, $W_4$ is an irreducible quadric surface in $\mathbb{P}^{3}$. As $D_2 \cdot (E_0-E_1-E_2) =0$, the result follows.
\end{proof}
Replacing $D_{d-2}$ with $tD_{d-2}$, a computation as in Equations~(\ref{rr1}) and (\ref{rr2}) and Serre vanishing shows that the Hilbert polynomial $\HP(S/I_{W_d},t)$ is equal to 
\begin{equation}\label{HPisHP}
\begin{array}{ccc}
\frac{((d-2)^2-|Y|)t^2+(3(d-2)-|Y|)t}{2}+1&=&\frac{d^2-5d+8}{4}t^2 - \frac{d^2-9d+12}{4}t+1.
\end{array}
\end{equation}
\section{The Wachspress quadrics}\label{sec:three}
In this section, we construct a set of quadrics which vanish 
on $W_d$. These quadrics are polynomials that are expressed
as a scalar product with a fixed vector $\tau$.  The vector $\tau$ defines a
linear projection $\p^{d-1}\dra\p^2$, also denoted by $\tau$, given
by $$\txbf{x}\longmapsto
\sum_{i=1}^{d}x_i{\bf v}_i,$$ where $\txbf{x}=[x_1:\cdots :x_{d}]\in\p^{d-1}$.
By the second property of barycentric coordinates,
the composition 
$\tau\circ w_d:\p^2\dra\p^2$ is the identity map on $\p^2$. Since ${\bf v}_i\in\kk^3$, the vector $\tau$ is a triple of linear forms 
$(\tau_1,\tau_2,\tau_3) \in S^3$. The linear
subspace $\mcC$ of $\p^{d-1}$ where the projection is undefined 
is the {\em center of projection}, 
and $I_{\mcC} =\langle \tau_1,\tau_2,\tau_3 \rangle$.

\subsection{Diagonal Monomials}
 A \emph{diagonal monomial} is a monomial
$x_ix_j \in S_2$ such that $j\notin\{i-1,i,i+1\}$. We write
$\mathcal{D}$ for the subspace of $S_2$ spanned by the
diagonal monomials; identifying $x_i$ with the vertex $v_i$, 
a diagonal monomial is a diagonal in $P_d$, see Figure \ref{diag}.
\vspace{10pt}
\begin{figure}[htp]
  \begin{center}
\psfrag{xi}{\Large$x_i$}
\psfrag{xj}{\Large$x_j$}
\includegraphics[scale=.5]{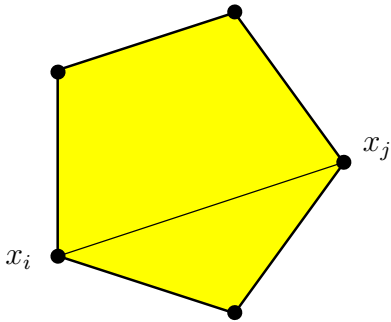}
  \end{center}
  \caption{A diagonal monomial}
\label{diag}
\end{figure}
\begin{lem}
Any quadric which vanishes on $W_d$ is a linear combination of 
elements of $\mathcal{D}$.
\label{squarefree}
\end{lem}
\begin{proof}
Let $Q$ be a polynomial in $(I_{W_d})_2$.  Then $Q(w_d)=Q(b_1,\dots,b_d)=0$.  On
the edge $[v_k,v_{k+1}]$ all the $b_i$ vanish except $b_k$ and $b_{k+1}$.  Thus
on this edge, the expression $Q(w_d)=0$ is 
\begin{equation}
\label{monomcombo}
c_1b_k^2+c_2b_kb_{k+1}+c_3b_{k+1}^2=0
\end{equation}
 for
some constants $c_1,c_2,$ and $c_3$ in $\kk$.  Recall that $b_i(v_j)=0$ if
$i\neq j$ and $b_i(v_i)\neq 0$ for each $i$. Evaluating  Equation
\ref{monomcombo} at $v_k$ and $v_{k+1}$, we conclude $c_1=c_3=0$.  At an
interior point of edge $[v_k,v_{k+1}]$ neither $b_k$ nor $b_{k+1}$
vanishes.  This implies that $c_2=0$.  A similar calculation on each edge shows
that all coefficients of non-diagonal terms in $Q$ are zero. 
 \end{proof}

\subsection{The Map to $(I_{\mcC})_2$}
We define a surjective map onto $(I_{\mcC})_2$, and use the map to
calculate the dimension of the vector space of polynomials 
in $(I_{\mcC})_2$ that are supported on diagonal monomials. 
Let $S_1^3$ denote the space of triples of linear forms on
$\p^{d-1}$.
Define the map $\Psi:S_1^3\to (I_{\mcC})_2$ by 
$F\mapsto F\cdot \tau,$ where $\cdot$ is the scalar product. 
\begin{lem}
The kernel of $\Psi$ is three-dimensional. 
\end{lem}
\begin{proof}
Since $I_{\mcC}$ is a complete intersection, the kernel is generated by 
the three Koszul syzygies on the $\tau_i$.
\end{proof}

Next we determine conditions on $F$ so that $\Psi(F) \in \mcd$. 
If ${\bf u}_i\in\kk^3$ for $i=1,\dots,d$, then $$F=\sum_{i=1}^d x_i{\bf u}_i$$
is an element of $S_1^3$.
Viewing the projection $\tau$ as an element of $S_1^3$ we have
\begin{equation}
\Psi(F)=F\cdot\tau=(\sum_{i=1}^d x_i{\bf u}_i)\cdot (\sum_{i=1}^d x_i{\bf v}_i)=
\sum_{i,j=1}^d ({\bf u}_i\cdot {\bf v}_j+{\bf u}_j\cdot {\bf v}_i)x_ix_j.
\end{equation}
If $\Psi(F) \in \mcd$ then the coefficients of non-diagonal monomials
must vanish:
\begin{equation}
\label{eq2}
 {\bf u}_i\cdot {\bf v}_i =0\ \ \ \ \mbox{and}\ \ \ \ 
 {\bf u}_i\cdot {\bf v}_{i+1}+{\bf u}_{i+1}\cdot {\bf v}_i =0 \ \ \ \mbox{for all } i.
\end{equation}
\begin{lem}
\label{dimensionComp}
The dimension of the vector space $\mcd\cap (I_{\mcC})_2$ is $d-3$.
\end{lem}
\begin{proof}
We show the conditions in Equation \eqref{eq2} give $2d$ independent conditions
on the $3d$-dimensional vector space $S_1^3$, and the solution
space is $\Psi^{-1}(\mcd \cap (I_{\mcC})_2)$, thus $\dim(\Psi^{-1}(\mcd
\cap (I_{\mcC})_2))=d$.
 The conditions are represented by the matrix equation:
$$
\left(
\begin{array}{c}
{\bf v}_1\cdot {\bf u}_1\\
\vdots\\
{\bf v}_d\cdot {\bf u}_d\\
{\bf v}_1 \cdot {\bf u}_2+{\bf v}_2 \cdot {\bf u}_1\\
\vdots\\
{\bf v}_d \cdot {\bf u}_1+{\bf v}_1 \cdot {\bf u}_d
\end{array}
\right)=
\overbrace{\left(\begin{array}{cccc}
{\bf v}_1^T & 0 & \cdots & 0\\
0 & {\bf v}_2^T &  & \vdots\\
\vdots &  & \ddots & 0\\
0 & \cdots & 0 & {\bf v}_d^T\\
{\bf v}_2^T & {\bf v}_1^T &  & 0\\
0 &  & \ddots & \\
{\bf v}_d^T &  &  & {\bf v}_1^T
\end{array}\right)}^M
\left(\begin{array}{c}
{\bf u}_1\\
{\bf u}_2\\
\vdots\\
\vdots\\
\vdots\\
\vdots\\
{\bf u}_d
\end{array}
\right)=
\left(
\begin{array}{l}
0\\
0\\
\vdots\\
\vdots\\
\vdots\\
0
\end{array}
\right),$$
where the ${\bf v}_i$ and ${\bf u}_i$ are column vectors the superscript $T$ indicates
transpose. The matrix $M$ in the middle is
a $2d\times 3d$ matrix, and the proof will be complete if the rows are shown to
be  independent.  Denote the rows of $M$ by $r_1,\dots,r_d,r_{d+1}\dots,r_{2d}$
and let  $c_1 r_1+\cdots+c_d r_d+c_{d+1}r_{d+1}+\cdots+c_{2d} r_{2d}$ be a dependence relation among them.  The first three columns give  the dependence 
relation $c_1 {\bf v}_1+c_{d+1}{\bf v}_2+c_{2d}{\bf v}_d=0$. 
Since ${\bf v}_d,{\bf v}_1,$ and ${\bf v}_2$ define adjacent rays of a polyhedral cone, they must
be
independent, so $c_1,c_{d+1},$ and $c_{2d}$ must be zero. 
Repeating the process at each triple ${\bf v}_{i-1},{\bf v}_i,$ and ${\bf v}_{i+1}$ shows the rest of the $c_i$'s vanish. Since the restriction 
$\Psi:\Psi^{-1}(\mcd \cap
(I_{\mcC})_2)\to \mcd \cap
(I_{\mcC})_2$ remains surjective we find $\dim(\mcd \cap
(I_{\mcC})_2)=\dim(\Psi^{-1}(\mcd \cap
(I_{\mcC})_2))-\dim(\ker(\Psi))=d-3$.
\end{proof}

\subsection{Wachspress Quadrics}
We now compute the dimension and a generating set for $(I_{W_d})_2$. 
\begin{defn}\label{gammaSet}
Let $\gamma(i)$ denote the set $\{1,\dots,d\}\setminus
\{i-1,i\}$, $\gamma(i,j)=\gamma(i)\cap\gamma(j)$, and 
$\gamma(i,j,k)=\gamma(i)\cap\gamma(j)\cap\gamma(k)$.
\end{defn}
The image of a diagonal monomial $x_ix_j$ under the pullback map
$w_d^*: S \to R$ is
$$b_ib_j=\alpha_i\alpha_j\prod_{k\in
\gamma(i)}\ell_k\prod_{m\in\gamma(j)}\ell_m=\alpha_i\alpha_j\prod_{k=1}^d
\ell_k
\prod_{m\in\gamma(i,j)}  \ell_m,$$
and each diagonal monomial has a common factor
$A=\prod_{k=1}^d\ell_k$.  
To find the quadratic relations among Wachspress coordinates it suffices to
find linear relations among
products $\prod_{m\in\gamma(i,j)}\ell_m\in R_{d-4}$ for
diagonal pairs $i,j$. Define the map $\phi:\mcd\to R_{d-4}$ by 
$x_ix_j\longmapsto \frac{b_i b_j}{A},$ and extend by linearity; this is 
$w_d^*$ restricted to $\mcd$ and divided by $A$.  By Lemma
\ref{squarefree} it follows that
$(I_{W_d})_2=\ker(\phi)\sset \mcd$.
\begin{lem}\label{surjmap}
The dimension of $(I_{W_d})_2$ is $d-3$.
\end{lem}
\begin{proof}
We will show $\phi:\mcd \to R_{d-4}$ is surjective with
$\dim(\ker \phi)=d-3$. To see this, note that there are $d-3$ 
diagonal monomials that have $x_1$ as a factor.  We show
that the images of the remaining
\begin{equation*}
d(d-3)/2-(d-3)=(d-3)(d-2)/2=\dim(R_{d-4})
\end{equation*}
diagonal monomials are independent. Let $p_{s,t}=\ell_s\cap\ell_t$ and
$x_{p,q} = x_px_q$. In Table \ref{prooftable}, a star, $*$,
represents a nonzero number, a blank space is zero.  The $(i,j)$ entry in 
Table \ref{prooftable} represents the value of the image of the diagonal
monomial in column $j$ at the external vertex in row $i$.  The external
vertices not lying on $\ell_d$ are arranged down the rows with their indices in
lexicographic order. 
\begin{table}[htp]
\begin{center}
$\begin{array}{cccccccccccc}
 &  x_{2,4}  &  \cdots &  x_{2,d} & x_{3,5} & \cdots &
 x_{3,d} & \cdots  &  x_{d-3, d-1}  &
 x_{d-3,d} & x_{d-2,d} & \\
 p_{1,3}  & * &  &  &  &  &  &  &  & \\
\vdots &  &  \ddots &  &  &  &  &  &  & \\
 p_{1,d-1}   &  &  & * &  &  &  &  &  & \\
 p_{2,4}  & * &  &  & * &  &  &  &  & \\
\vdots&  &  &  &  & \ddots &  &  &  & \\
 p_{2,d-1}  &  &  & * &  &  & * &  &  & \\
\vdots&  &  &  &  &  &  & \ddots &  & \\
 p_{(d-4)(d-2)}  &  &  &  &  &  &  &  & * & \\
 p_{(d-4)(d-1)}  &  &  &  &  &  &  &   & * & * & \ \ \\
 {p_{(d-3)(d-1)}}  &  &  &  &  &  &  &  & * & * & * \ \ 
\end{array}$
\vskip .2in
\caption{Values of images of diagonal monomials at external vertices}
\label{prooftable}
\end{center}
\end{table}

Since Table~\ref{prooftable} is lower triangular, the images are 
independent. We have found $\dim(R_{d-4})$ independent images and hence
$\phi$ is surjective.  This is a map from
a vector space of dimension $d(d-3)/2$ to one of dimension $(d-2)(d-3)/2$. 
The map is surjective, so the kernel has
dimension $d(d-3)/2-(d-2)(d-3)/2=d-3$.
\end{proof}

There is a generating set for $(I_{W_d})_2$ where each generator is a scalar
product with the vector $\tau$.  The other vectors in these scalar products
are $$\Lambda_k=\frac{x_{k+1}}{\alpha_{k+1}} \bn_{k+1}-\frac{x_k}{\alpha_k}
 \bn_{k-1}\in S_1^3.$$
\begin{lem}
\label{basis}
The vectors $\{\Lambda_1\dots,\Lambda_d\}$ form a basis for the space
$\Psi^{-1}(\mcd\cap (I_{\mcC})_2)$.
\end{lem}
\begin{proof}
Suppose that $\sum_{k=1}^dc_k\Lambda_k=0$ is a linear
dependence relation
among the $\Lambda_k.$ The coefficient of a variable $x_k$ is 
$$ \frac{1}{\alpha_{k}}(c_{k-1} \bn_k-c_k \bn_{k-1}).$$ By the
dependence relation this must be zero, which implies that $ \bn_{k-1}$ and
$ \bn_k$ are scalar multiples.  This is impossible since they
are normal vectors of adjacent facets of a polyhedral cone.  Hence,
$c_{k-1}=c_k=0$ for all $k$ which shows that the $\Lambda_k$ are independent. 

In the proof of Lemma \ref{dimensionComp} we showed that
$\dim(\Psi^{-1}(\mcd\cap
(I_{\mcC})_2))=d$ and we have just shown $\dim(\langle \Lambda_k\mid
k=1,\dots,d\rangle)=d$. To prove the result, it suffices to show 
$\langle \Lambda_k\mid
k=1,\dots,d\rangle\sset \Psi^{-1}(\mcd\cap (I_{\mcC})_2)$.
The conditions of Equation \eqref{eq2} are required for
$\Lambda_k \in S_1^3$ to lie in $\Psi^{-1}(\mcd\cap (I_{\mcC})_2)$.
We show these conditions are satisfied for each $\Lambda_k$. 
 
Let ${\bf u}_i=0$ if $i\neq k,k+1$, 
${\bf u}_k=- \bn_{k-1}/\alpha_k$, and ${\bf u}_{k+1}= \bn_{k+1}/\alpha_{k+1}$
for each fixed $k$.  Then
$$\Lambda_k=\frac{x_{k+1}}{\alpha_{k+1}} \bn_{k+1}-\frac{x_k}{\alpha_k}
 \bn_{k-1}=\sum_{i=1}^d{\bf u}_ix_i.$$
Since $\bn_{k-1}\cdot {\bf v}_k=0$, $\bn_{k+1}\cdot {\bf v}_{k+1}=0$, and ${\bf u}_i=0$ for $i\neq
k,k+1$ we have that ${\bf u}_i\cdot {\bf v}_i=0$ for
each $i=1,\dots d$. The expression ${\bf u}_i\cdot {\bf v}_{i+1}+{\bf u}_{i+1}\cdot {\bf v}_i$ is
zero for all $i\neq k-1,k,k+1$ simply because ${\bf u}_i=0$ for $i\neq k,k+1$. 
We have
\begin{eqnarray*}
{\bf u}_k\cdot {\bf v}_{k+1}+{\bf u}_{k+1}\cdot {\bf v}_k&=&-\frac{ \bn_{k-1}}{\alpha_k}\cdot
{\bf v}_{k+1}+\frac{ \bn_{k+1}}{\alpha_{k+1}} \cdot {\bf v}_k\\
 &=& -\frac{{\bf v}_{k-1}\times
{\bf v}_k\cdot
{\bf v}_{k+1}}{\alpha_k}+\frac{{\bf v}_{k+1}\times {\bf v}_{k+2}\cdot
{\bf v}_k}{
\alpha_{k+1}} \\
&=&-\frac{ |{\bf v}_{k-1} {\bf v}_k {\bf v}_{k+1}|}{\alpha_k}+\frac{|{\bf v}_{k+1} {\bf v}_{ k+2 } {\bf v}_k| } {\alpha_{k+1}}=0,
\end{eqnarray*}
as $\alpha_j=|{\bf v}_{j-1} {\bf v}_j {\bf v}_{j+1}|$.  It is easy to show
that the expression ${\bf u}_i\cdot {\bf v}_{i+1}+{\bf u}_{i+1}\cdot {\bf v}_i$ is zero for
$i=k\pm 1$. Thus the ${\bf u}_i$ satisfy the
conditions in Equation~\eqref{eq2}, so 
$\Lambda_k\in\Psi^{-1}(\mcd\cap (I_{\mcC})_2)$.
\end{proof}

\begin{thm}\label{Wquads}\txbf{(Wachspress Quadrics)}\\
The Wachspress quadrics $(I_{W_d})_2$ are those elements of $S_2$ which 
are diagonally supported and vanish on $\mathcal{C}$. The 
quadrics $Q_k=\Lambda_k\cdot \tau$ for $k=1,\dots,d$ span $(I_{W_d})_2$. 
\end{thm}
\begin{proof}
Let $\txbf{p}$ be the vector $(x,y,z)$.  By definition of Wachspress
coordinates,
$$ \tau(w_d(\txbf{p}))=\sum_{i=1}^d b_i(\txbf{p}){\bf v}_i= 
\txbf{p} \sum_{i=1}^d b_i(\txbf{p}).$$
We have
{\allowdisplaybreaks\begin{align*}
\Lambda_k(w_d(\txbf{p}))&=\frac{b_{k+1}(\txbf{p})}{\alpha_{k+1}} \bn_{k+1} -\frac{b_{k}(\txbf{p})}{\alpha_{k}} \bn_{k-1}\\
&=(\prod_{j\neq k,k+1}\ell_j)\bn_{k+1} - (\prod_{j\neq k-1,k}\ell_j)\bn_{k-1}\\
&=(\prod_{j\neq k-1,k,k+1}\ell_j)(\ell_{k-1} \bn_{k+1} - \ell_{k+1} \bn_{k-1})\\
&= H  [ \bn_{k+1}\ ( \bn_{k-1}\cdot\txbf{p})
- \bn_{k-1}\ ( \bn_{k+1}\cdot\txbf{p})],
\end{align*}}
where $H= \displaystyle{\prod_{j\neq k-1,k,k+1}\ell_j}$. Set $\overline{H}:=H\ \sum_{i=1}^d b_i(\txbf{p})$.
Then we have
{\allowdisplaybreaks\begin{align*}
Q_k(w_d(\txbf{p}))
&=\tau(w_d(\txbf{p}))\cdot\Lambda_k(w_d(\txbf{p}))\\&=
\overline{H}\ \txbf{p}\cdot
[ \bn_{k+1}\ ( \bn_{k-1}\cdot\txbf{p})
- \bn_{k-1}\ ( \bn_{k+1}\cdot\txbf{p})]\\
&=
\overline{H}\
[(  \bn_{k+1}\cdot\txbf{p} )( \bn_{k-1}\cdot\txbf{p})
-( \bn_{k-1}\cdot\txbf{p} )( \bn_{k+1}\cdot\txbf{p})]=0.
\end{align*}}
We have just shown that $ Q_k\in (I_{W_d})_2$. By Lemma~\ref{basis} $\Psi^{-1}(\mcd\cap (I_{\mcC})_2)$ is spanned by the $\Lambda_k$. Observe that
$\langle Q_1,\dots,Q_d\rangle = \Psi(\langle \Lambda_k\rangle)=\mcd\cap
(I_{\mcC})_2$. Thus $\dim(\langle Q_1\dots,Q_d\rangle)=d-3$, 
and by Lemma~\ref{surjmap} $\dim((I_{W_d})_2)=d-3$. 
Therefore, since $\langle Q_1\dots,Q_d\rangle\sset (I_{W_d})_2$, we 
have $\langle Q_1,\dots,Q_d\rangle=(I_{W_d})_2=\mcd\cap (I_{\mcC})_2$.
\end{proof}

\begin{cor}\label{inQuadrics}
The quadrics $\{ \Lambda_2 \cdot \tau,\ldots, \Lambda_{d-2} \cdot \tau \}$ are a
basis for the quadrics in $I_{W_d}$, and
in graded lex order, $\{x_1x_3,\ldots, x_1x_{d-1}\}$ is a basis for
$in_{\prec}(I_{W_d})_2$.
\end{cor}
\begin{proof}
Expanding the expression for $\Lambda_i \cdot \tau$ yields
\[
\Lambda_i \cdot \tau = x_1 x_{i+1}(\frac{{\bf v}_1 \cdot {\bf
n}_{i+1}}{\alpha_{i+1}}) -  x_1 x_{i}(\frac{{\bf v}_1 \cdot {\bf
n}_{i-1}}{\alpha_{i}})+ \zeta_i,
\]
where $\zeta_i \in \kk[x_2,\ldots,x_d]$. Since ${\bf n}_i = {\bf v}_i \times {\bf v}_{i+1}$, 
\[
\Lambda_2 \cdot \tau = x_1 x_3(\frac{{\bf v}_1 \cdot {\bf n}_3}{\alpha_3})+
\zeta_2.
\]
Since no three of the lines $\mathbb{V}(l_i)$ are concurrent, ${\bf v}_i \cdot {\bf n}_j$ is nonzero unless $j \in \{i,i+1\}$, so we may use the lead term of
$\Lambda_2 \cdot \tau$ to reduce $\Lambda_3 \cdot \tau$ to $x_1x_4 +
f(x_2,\ldots,x_d)$. Repeating the process proves
that 
\[
\{x_1x_3,\ldots, x_1x_{d-1}\} \subseteq in_{\prec}(I_{W_d})_2.
\]
By Lemma~\ref{surjmap},  $(I_{W_d})_2$ has dimension $d-3$, which
concludes the proof.
\end{proof}

\begin{cor}\label{noLsyz}
There are no linear first syzygies on $(I_{W_d})_2$.
\end{cor}
\begin{proof}
By Corollary~\ref{inQuadrics}, we may assume that a basis for $(I_{W_d})_2$ 
has the form 
\vskip .1in
\begin{center}
$\begin{array}{ccc}
x_1x_3 & + & \zeta_3(x_2,\ldots,x_d)\\
x_1x_4 & + & \zeta_4(x_2,\ldots,x_d)\\
x_1x_5 & + & \zeta_5(x_2,\ldots,x_d)\\
\vdots& + & \vdots\\
x_1x_{d-1} & + & \zeta_{d-1}(x_2,\ldots,x_d).
\end{array}$
\end{center}
\vskip .1in
Since the $\zeta_i$ do not involve $x_1$, this implies that any linear first syzygy on $(I_{W_d})_2$ 
must be a linear combination of the Koszul syzygies on $\{x_3,\ldots,x_{d-1}\}$. Now change the
term order to graded lex with $x_i > x_{i+1} > \cdots > x_d > x_1 >x_2 \cdots > x_{i-1}$. In this order,
arguing as in the proof of Corollary~\ref{inQuadrics} shows that we may 
assume a basis for $(I_{W_d})_2$ has the 
form
\vskip .1in
\begin{center}
$\begin{array}{ccc}
x_ix_{i+2} & + & \zeta_{i+2}(x_1,\ldots,\widehat{x_i}, \ldots,x_d)\\
x_ix_{i+3} & + & \zeta_{i+3}(x_1,\ldots,\widehat{x_i}, \ldots,x_d)\\
x_ix_{i+4} & + & \zeta_{i+4}(x_1,\ldots,\widehat{x_i}, \ldots,x_d)\\
\vdots& + & \vdots\\
x_ix_{i-2} & + & \zeta_{i-2}(x_1,\ldots,\widehat{x_i}, \ldots,x_d).
\end{array}$
\end{center}
\vskip .1in
Hence, any linear first syzygy on $(I_{W_d})_2$ must be a combination of Koszul syzygies on 
$x_{i+2},x_{i+3}, \ldots, x_{i-2}$. Iterating this process for the term orders above shows
there can be no linear first syzygies on $(I_{W_d})_2$.
\end{proof}
\subsection{Decomposition of $\V(\langle (I_{W_d})_2 \rangle)$}
We now prove that $\V(\langle (I_{W_d})_2 \rangle) = \mathcal{C} \bigcup W_d$.
The results in \S 4 and \S 5 are independent of this fact. 
\begin{lem}
\label{normformulas}
 For any $i,j,$ and $k$ we have
$$| \bn_i\  \bn_j\  \bn_k|=|{\bf v}_j\ {\bf v}_k\ {\bf v}_{k+1}|\cdot |{\bf v}_i\ {\bf v}_{i+1}\
{\bf v}_{j+1}|-|{\bf v}_{j+1}\
{\bf v}_k\ {\bf v}_{k+1}|\cdot |{\bf v}_i\
{\bf v}_{i+1}\ {\bf v}_j|$$
\end{lem}
\begin{proof}
Apply the formulas
$\txbf{a}\times (\txbf{b}\times \txbf{c})=\txbf{b}(\txbf{a}\cdot
\txbf{c})-\txbf{c}(\txbf{a}\cdot \txbf{b})$ and $|\txbf{a}\ \txbf{b} \
\txbf{c}|=\txbf{a}\times \txbf{b}\cdot \txbf{c}$,
\begin{eqnarray*}
| \bn_i\  \bn_j\  \bn_k|&=& \bn_i\times  \bn_j\cdot
 \bn_k
= ( \bn_i\times ({\bf v}_j\times {\bf v}_{j+1}))\cdot \bn_k\\
&=& [{\bf v}_j( \bn_i\cdot {\bf v}_{j+1})-{\bf v}_{j+1}( \bn_i\cdot {\bf v}_j)]\cdot
 \bn_k\\
&=&
({\bf v}_j\cdot  \bn_k)( \bn_i\cdot {\bf v}_{j+1})-({\bf v}_{j+1}\cdot
 \bn_k)( \bn_i\cdot {\bf v}_j) \\
&=&
 |{\bf v}_j\ {\bf v}_k\ {\bf v}_{k+1}|\cdot |{\bf v}_i\ {\bf v}_{i+1}\
{\bf v}_{j+1}|
-|{\bf v}_{j+1}\
{\bf v}_k\ {\bf v}_{k+1}|\cdot |{\bf v}_i\
{\bf v}_{i+1}\ {\bf v}_j|.
\end{eqnarray*}
\end{proof}
\begin{cor}
\label{normformcor}
 $$| \bn_i\  \bn_j\  \bn_{j+1}|=\alpha_{j+1}|{\bf v}_i\ {\bf v}_{i+1}\
{\bf v}_{j+1}|$$
\end{cor}
\begin{proof}
This follows from Lemma \ref{normformulas} and the definition of $\alpha_{j+1}$.
\end{proof}
\begin{cor}
$$ | \bn_{i-1}\  \bn_i\  \bn_{i+1}|=\alpha_i\alpha_{i+1}$$
\end{cor}
\begin{proof}
This follows from Lemma \ref{normformulas} and the definition of $\alpha_{i}$
and $\alpha_{i+1}$.
\end{proof}
\begin{lem}
\label{decomplemma}
Let ${\bf x}=[x_1:\cdots:x_d]\in\V(\langle (I_{W_d})_2 \rangle)\setminus\mcC$.
If $\tau({\bf x})$ is a base point $p_{ij}=\bn_i\times \bn_j$, then ${\bf x}$ lies on the
exceptional line $\hat{p}_{ij}$ over $p_{ij}$. 
\end{lem}
\begin{proof}
Since indices are cyclic we assume that $i=1$.  Thus
$\tau({\bf x})=p_{1,j}=\bn_1\times \bn_j$ for some $j\notin\{d,1,2\}$.  The relation
$Q_1({\bf x})=\Lambda_1\cdot\tau({\bf x})=\Lambda_1\cdot (\bn_1\times \bn_j)=0$ yields 
\begin{equation}
L(1):= x_2\ \bn_2\cdot p_{1,j}-x_1\ \bn_d\cdot p_{1,j}=0.
\end{equation}
The relation $Q_j({\bf x})=0$ implies
\begin{equation}
 L(j):=x_{j+1}\ |\ \bn_{j+1}\ \bn_1\ \bn_j |- x_j\ |\ \bn_2\ \bn_1\ \bn_j\ |=0.
\end{equation}
Also,
\begin{align*}
 Q_2({\bf x})=(x_3\bn_3-x_2\bn_1)\cdot \bn_1\times \bn_j
=x_3\ |\ \bn_3\ \bn_1\ \bn_j\ |=0,
\end{align*}
implying $x_3=0$ since $|\ \bn_3\ \bn_1\ \bn_j\ |\neq 0$ if $j\neq 3$.
Assume $x_k=0$ for $3\leq k<j-1$.  Note that 
\begin{align*}
 Q_k({\bf x})=  
 (x_{k+1}\bn_{k+1}-x_k\bn_{k-1})\cdot \bn_1\times \bn_j
=x_{k+1}\ |\ \bn_{k+1}\ \bn_1\ \bn_j\ |=0,
\end{align*}
hence $x_{k+1}=0$ since $|\ \bn_{k+1}\ \bn_1\ \bn_j\ |\neq 0$ and by induction
$x_k=0$ for $3\leq k\leq j-1$. 
An analogous argument shows that $x_k=0$ for $j+2\leq k\leq d$.  Hence ${\bf x}$
lies on the line $\V(L(1),L(j),x_k\mid k\notin\{ 1,2,j,j+1 \})$, which is the
exceptional line $\hat{p}_{1,j}$.
\end{proof}
\begin{thm}\label{quadsNogen}
\label{Notcenter}
The subset $\V(\langle (I_{W_d})_2 \rangle)\setminus\mcC$ is contained in $\mcW$.  It
follows that the
variety $\V(\langle (I_{W_d})_2 \rangle)$ has irreducible decomposition $\mcW\cup
\mcC$.
\end{thm}
\begin{proof}
Let ${\bf x}=[x_1:\cdots:x_d]\in\V(\langle (I_{W_d})_2
\rangle)\setminus\mcC$.
The Wachspress quadrics give the relations
\begin{align}
\label{qrelations}
x_{r+1}\ \bn_{r+1}\cdot\tau=x_r\ \bn_{r-1}\cdot\tau 
\end{align}
for each $r=1,\dots d$.
\vspace{10pt}
\begin{figure}[htp!]
 \centering
\psfrag{k-2}{\Large$\bn_{k-2}$}
\psfrag{k-1}{\Large$\bn_{k-1}$}
\psfrag{k}{\Large$\bn_k$}
\includegraphics[scale=.27]{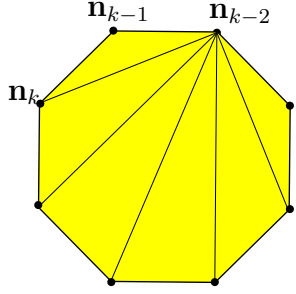}
\caption{Triangulation used for adjoint}
\label{triForAdj}
\end{figure} 
By Theorem~\ref{independent}, the adjoint is independent of triangulation, 
so we use $\mcA$ to denote the adjoint, specifying the triangulation if necessary. We now show for each $k\in \{1,\dots,d\}$, 
$b_k(\tau({\bf x}))=\mcA(\tau({\bf x}))\,x_k$, where the
triangulation above is used for the adjoint $\mcA$.
It follows from the uniqueness of Wachpress coordinates that the denominator 
$\sum_{i=1}^d b_i$ of $\beta_i$ is the adjoint of ${\bf P}_d^*$, so it follows that 
\begin{equation}
\label{gg}
w_d(\tau({\bf x}))=\mcA(\tau({\bf x}))\,{\bf x}.
\end{equation}
Provided $\mcA(\tau({\bf x}))\neq 0$, the result follows since
$w_d(\tau({\bf x}))\in \p^{d-1}$ is a
nonzero scalar multiple of ${\bf x}$, hence ${\bf x}$ is in
the image of the Wachspress map and thus lies on $\mcW$. If ${\bf x}\in\V(\langle (I_{W_d})_2 \rangle)\setminus\mcC$
and $\mcA(\tau({\bf x}))=0$, then by Equation \eqref{gg}  $w_d(\tau({\bf x}))=0$ and hence
$\tau({\bf x})$ is a basepoint of $w_d$.  Thus $\tau({\bf x})=\bn_i\times \bn_j$
for some diagonal pair $(i,j)$.  By Lemma \ref{decomplemma} ${\bf x}$ lies on an
exceptional line and hence lies on $\mcW$. To prove the claim, note
that since all indices are cyclic it suffices to assume
$k=3$. Let $|\bn_i \bn_j \bn_k| = |\bn_{ijk}|$ and 
\begin{equation*}
\bn_{i_1,\dots,i_m}\cdot\tau:=\prod_{j=1}^m (\bn_{i_j}\cdot\tau).
\end{equation*}
This is the product of $m$ linear forms in $S$, and
with this notation 
\[
b_3(\tau)=\bn_{1,4,5,\dots,d}\cdot\tau.
\]
For each $r\in\{3,\dots,d\}$ define 
\begin{eqnarray*}
\sigma_r:&=& (\bn_{4,\dots,r}\cdot\tau)\
\bn_1\cdot\bigg[\sum_{i=3}^{r}{\bf v}_{i}(\bn_{r+1,\dots,
d}\cdot\tau)\ x_{i}+\\
&&\ \ \ \sum_{i=r+1}^{d}{\bf v}_{i}(\bn_{r-1,\dots,i-2 } \cdot
\tau)\ (\bn_{i+1,\dots,d} \cdot\tau)\ x_r\bigg],
\end{eqnarray*}
where we set $\bn_{i,\dots,j}\cdot\tau =1$ if $j<i$.  We show
$x_3 \mcA(\tau({\bf x}))=\sigma_3=\sigma_d=b_3(\tau({\bf x}))$.
First, we show $\sigma_3=x_3\,\mcA(\tau)$: to see this, note that 
\begin{align}\label{deceq}
x_3 \mcA(\tau) = |\bn_{123}|\ (\bn_{4\cdots
d}\cdot\tau)\ x_3+\!\!\ \sum_{i=4}^{d}|\bn_{1,i-1,i}|\ (\bn_{2,\dots,i-2} \cdot \tau)\
(\bn_{i+1,\dots,d} \cdot\tau)\ x_3,
\end{align}
where we express the adjoint $\mcA$ using the triangulation in Figure
\ref{triForAdj}.  Applying the scalar triple product to $|\bn_{123}|$ and $|\bn_{1,i-1,i}|$ in the expression \eqref{deceq} yields, 
\begin{align}
\label{deceq2}
\bn_1\!\cdot\!(\bn_2\times \bn_3)\,(\bn_{4\cdots
d}\cdot\tau)\,x_3 \!+\!
 \sum_{i=4}^{d}\!\bn_1\cdot(\bn_{i-1}\times
\bn_{i})(\bn_{2,\dots,i-2}\cdot
\tau)\,(\bn_{i+1,\dots,d} \cdot\tau) x_3.
\end{align}
Factoring an $\bn_1$ and noting that $\bn_i\times
\bn_{i+1}={\bf v}_{i+1}$, \eqref{deceq2} becomes
\begin{align*}
\bn_1\cdot\bigg[{\bf v}_3(\bn_{4\cdots
d}\cdot\tau)\ x_3+ 
\,\sum_{i=4}^{d}{\bf v}_{i}(\bn_{2,\dots,i-2}\cdot
\tau) (\bn_{i+1,\dots,d} \cdot\tau)\,x_3\bigg]=\sigma_3.
\end{align*}
Now we show $\sigma_d=b_3(\tau)$.  Since $\bn_{d+1,\dots,d}\cdot\tau=1$
{\begin{align}
\label{deceq3}
\sigma_d=(\bn_{4,\dots,d}\cdot\tau) \bn_1\cdot\bigg(\sum_{i+3}^d {\bf v}_i
(\bn_{d+1,\dots,d}\cdot\tau) x_i\bigg)
=&
(\bn_{4,\dots,d}\cdot\tau) \bn_1\cdot\bigg(\sum_{i+3}^d {\bf v}_i
x_i\bigg).
\end{align}
Observing that
$\displaystyle{\bn_1\cdot\sum_{i=1}^2x_i{\bf v}_i=0}$ we see that \eqref{deceq3} is
\begin{align*}
(\bn_{4,\dots,d}\cdot\tau)\ 
(\bn_1\cdot\tau)
= \bn_{1,4,\dots,d}\cdot\tau=b_3(\tau).
\end{align*}}
We now claim that for $r\in\{3,\dots,d-1\}$ we have $\sigma_r=\sigma_{r+1}$.
Indeed,
\begin{align}
\sigma_r=&\,(\bn_{4,\dots,r}\cdot\tau)\,\bn_1\cdot\bigg[\sum_{i=3}^{r}{\bf v}_{i}(\bn_
{ r+1 ,\dots,
d}\cdot\tau)\ x_{i}+\\ \nonumber
&\,\sum_{i=r+1}^{d}{\bf v}_{i}(\bn_{r,\dots,i-2 } \cdot
\tau)\,(\bn_{i+1,\dots,d} \cdot\tau) (\bn_{r-1}\cdot\tau)\ x_r\,\bigg]\\ 
\nonumber
&=(\bn_{4,\dots,r}\cdot\tau)\,\bn_1\cdot\bigg[\sum_{i=3}^{r}{\bf v}_{i}(\bn_{r+1,\dots
,
d}\cdot\tau)\,x_{i}+\\ \nonumber
&\,\sum_{i=r+1}^{d}{\bf v}_{i}(\bn_{r,\dots,i-2 } \cdot
\tau)\,(\bn_{i+1,\dots,d} \cdot\tau)\,(\bn_{r+1}\cdot\tau)\ x_{r+1}\bigg],
\end{align}
where we have applied \eqref{qrelations} to the last term. Factoring 
out $\bn_{r+1}\cdot\tau$ yields
\begin{align*}
&\,(\bn_{4,\dots,r+1}\cdot\tau)\,\bn_1\cdot\bigg[\sum_{i=3}^{r}{\bf v}_{i}(\bn_{r+2,
\dots,
d}\cdot\tau)\,x_{i}+
 \sum_{i=r+1}^{d}{\bf v}_{i}(\bn_{r,\dots,i-2 } \cdot
\tau)\,(\bn_{i+1,\dots,d} \cdot\tau)\,x_{r+1}\bigg] 
\end{align*}
Lastly, since the expressions in both summations agree at the index $i=r+1$ we
can shift the indices of summation,
\begin{align*}
&\ 
(\bn_{4,\dots,r+1}\cdot\tau)\,\bn_1\cdot
\bigg[\sum_{i=3}^{r+1}{\bf v}_{i}(\bn_{r+2,\dots,
d}\cdot\tau)\ x_{i}+
 \sum_{i=r+2}^{d}{\bf v}_{i}(\bn_{r,\dots,i-2 } \cdot
\tau)\,(\bn_{i+1,\dots,d} \cdot\tau)\,x_{r+1}\bigg],
\end{align*}
which is precisely $\sigma_{r+1}$, proving the claim. The claim shows that
$\sigma_3=\sigma_{d}$, hence \eqref{gg} holds and so ${\bf x}$ lies in
$\mcW$ if $\mcA(\tau({\bf x}))\neq 0$. 
\end{proof}
\section{The Wachspress cubics}\label{sec:four}
Theorem~\ref{quadsNogen} shows that the Wachspress quadrics do not suffice to
cut out the Wachspress variety $\mcW$.  We now construct cubics, the
\emph{Wachspress cubics}, that lie in $I_{W_d}$ and do not arise from 
the Wachspress quadrics. These cubics are determinants
of $3\times 3$ matrices of linear forms. The key to showing that they are 
in $I_{W_d}$ is to write them as a difference of adjoints 
$\mathcal{A}_{T_1(C)} - \mathcal{A}_{T_2(C)}$, where $T_1(C)$ and $T_2(C)$ are
two different triangulations of a subcone $C$ of the dual cone 
${\bf P}_d^*$. By Theorem~\ref{independent}, the difference is zero, 
so the cubic is in $I_{W_d}$.
 
\subsection{Construction of Wachspress Cubics}
As in Lemma~\ref{basis}, let
$$\Lambda_r=\frac{x_{r+1}}{\alpha_{r+1}} \bn_{r+1}-\frac{x_r}{\alpha_r} \bn_{r-1}.$$ 
\begin{thm}
\label{cubicgens}
If $i \ne j \ne k \ne i$, then $w_{i,j,k}:=|\Lambda_i,\Lambda_j,\Lambda_k| \in I_{W_d}$.
\end{thm}
\begin{proof}
We break the proof into two parts. First, suppose 
no pair of $(i,j,k)$ corresponds to an edge 
of $P_d$. We call such an $(i,j,k)$ a $T$-triple.
A direct calculation shows that if $(i,j,k)$ is a $T$-triple, 
then evaluating the monomial $x_ix_jx_k$ at Wachspress coordinates yields
\begin{equation}
x_ix_jx_k(w_d)=b_i b_j b_k=A^2\prod_{m\in\gamma(i,j,k)} \ell_m,
\end{equation}
where $\gamma(i,j,k)$ is as in Definition~\ref{gammaSet}.
Since there are no $T$-triples if $d< 6$, we may assume $d\geq 6$. 
Changing variables by replacing $x_i$ with $x_i/\alpha_i$, we may ignore
the constants $\alpha_i$. Using the definition of the $\Lambda$'s, observe
that $w_{i,j,k}=$
\begin{equation}\label{PCeqn}
\begin{array}{cccc}
|\bn_{i+1}\ \bn_{j+1}\ \bn_{k+1}|x_{i+1}x_{j+1}x_{k+1} &- &|\bn_{i+1}\
\bn_{j+1}\ \bn_{k-1}|x_{i+1}x_{j+1}x_{k} &-\\
\!\!\!\!\!\! |\bn_{i+1}\ \bn_{j-1}\ \bn_{k+1}|x_{i+1}x_{j}x_{k+1}
&+&\!\!\!\!\!\! |\bn_{i+1}\ \bn_{j-1}\ \bn_{k-1}|x_{i+1}x_{j}x_{k}&-\\
\!\!\!\!\!\! |\bn_{i-1}\ \bn_{j+1}\ \bn_{k+1}|x_{i}x_{j+1}x_{k+1}
&+&\!\!\!\!\!\! |\bn_{i-1}\ \bn_{j+1}\ \bn_{k-1}|x_{i}x_{j+1}x_{k}&+\\
\!\!\!\! \!\!\!\!\!\!\!\! |\bn_{i-1}\ \bn_{j-1}\ \bn_{k+1}|x_{i}x_{j}x_{k+1}
&-&\!\!\!\! \!\!\!\!\!\!\!\! |\bn_{i-1}\ \bn_{j-1}\ \bn_{k-1}|x_{i}x_{j}x_{k}.&
\end{array}
\end{equation}
There are several situations to consider, depending on various 
possibilities for interactions among the indices. Interactions may occur 
if $i+1=j-1$ or $j+1=k-1$ or $k+1=i-1$, so there are four cases:
\begin{center}
$\begin{array}{cccc}
\mbox{1. All three hold } & \mbox{2. Two hold } &  \mbox{3. One holds }&
\mbox{4. None hold}.
\end{array}$
\end{center}

\txbf{Case 1:}
The indices $(i,j,k)$ satisfy Case 1 if and only if $d=6$.  For
$d=6$ there are only two $T$-triples; $(1,3,5)$ and $(2,4,6)$.
We show that $w_{1,3,5}$ vanishes on Wachspress coordinates; the 
case of $w_{2,4,6}$ is similar. All but two of the determinants 
in Equation~\eqref{PCeqn} vanish, leaving
\begin{align}
\label{theabove}
w_{1,3,5}=|\Lambda_1,\Lambda_3,\Lambda_5|=|\bn_{2}\ \bn_{4}\
\bn_{6}|x_{2}x_{4}x_{6}-|\bn_{6}\ \bn_{2}\
\bn_{4}|x_{1}x_{3}x_{5}.
\end{align}
Notice that the coefficients are equal, and we conclude by showing
that 
\[
x_1x_3x_5-x_2x_4x_6
\]
vanishes on Wachspress coordinates.  The monomials
$x_1x_3x_5$ and $x_2x_4x_6$ evaluated at Wachspress coordinates are $b_1b_3b_5$
and $b_2b_4b_6$, respectively. Both of these are equal to $A^2$, so
$x_1x_3x_5-x_2x_4x_6$ vanishes on Wachspress coordinates.\\

\pagebreak
\txbf{Case 2:}  We can assume without loss of generality $i+1\neq j-1,$
$j+1=k-1$, and $k+1=i-1$. Four coefficients vanish in Equation~\ref{PCeqn}, yielding
\begin{eqnarray*}
 w_{ijk}&=&|\bn_{i+1}\ \bn_{j+1}\ \bn_{i-1}|x_{i+1}x_{j+1}x_{i-1}\\
       &-&|\bn_{i+1}\ \bn_{j-1}\ \bn_{i-1}|x_{i+1}x_{j}x_{i-1}\\
       &+&|\bn_{i+1}\ \bn_{j-1}\ \bn_{j+1}|x_{i+1}x_{j}x_{i-2}\\
       &-&|\bn_{i-1}\ \bn_{j-1}\ \bn_{j+1}|x_{i}x_{j}x_{i-2}.
\end{eqnarray*}
Evaluating this at Wachspress coordinates yields,
\begin{eqnarray*}
\label{Quadeval}
w_{ijk}\circ w_d&=&
|\bn_{i+1} \bn_{j+1} \bn_{i-1} |\prod_{m\in\gamma(i+1,j+1,i-1)}\ell_m+
|\bn_{i+1}
\bn_{j-1} \bn_{i-1}|\prod_{m\in\gamma(i+1,j,i-1)} \ell_m-\\
&&|\bn_{i+1} \bn_{j-1} \bn_{j+1}|\prod_{m\in\gamma(i+1,j,i-1)}\ell_m  -
|\bn_{i-1}
\bn_{j-1} \bn_{j+1}|\prod_{m\in\gamma(i,j,i-1)}\ell_m\\
&=&A^2\ \bigg( \prod_{m\in\gamma(i-1,i+1,j+1,j)}\ell_m \bigg) (|\bn_{i+1}
\bn_{j+1}
\bn_{i-1}
|\ell_{j-1}-|\bn_{i+1}\
\bn_{j-1}\ \bn_{i-1}|\ell_{j+1}+\\
&&|\bn_{i+1}\ \bn_{j-1}\ \bn_{j+1}|\ell_{i-1}-|\bn_{i-1}
\bn_{j-1} \bn_{j+1}|\ell_{i+1})\\
&=&A^2\ \bigg( \prod_{m\in\gamma(i-1,i+1,j+1,j)}\ell_m \bigg) 
\bigg[ \big(|\bn_{i+1} \bn_{j+1} \bn_{i-1}|\ell_{j-1}+
|\bn_{i-1} \bn_{j+1} \bn_{j-1}|\ell_{i+1}\big)- \\
&&\big(|\bn_{i+1}\ \bn_{j-1}\ \bn_{i-1}|\ell_{j+1}+
|\bn_{i+1}\ \bn_{j+1}\ \bn_{j-1}|\ell_{i-1}\big)\bigg], \mbox{ where }A=\prod_{i=1}^d\ell_i.
\end{eqnarray*}
\begin{figure}[htp]
\centering
\psfrag{v1}{\Large$\bn_{i+1}$}
\psfrag{v2}{\Large$\bn_{j+1}$}
\psfrag{v3}{\Large$\bn_{j-1}$}
\psfrag{v4}{\Large$\bn_{i-1}$}
\includegraphics[scale=.6]{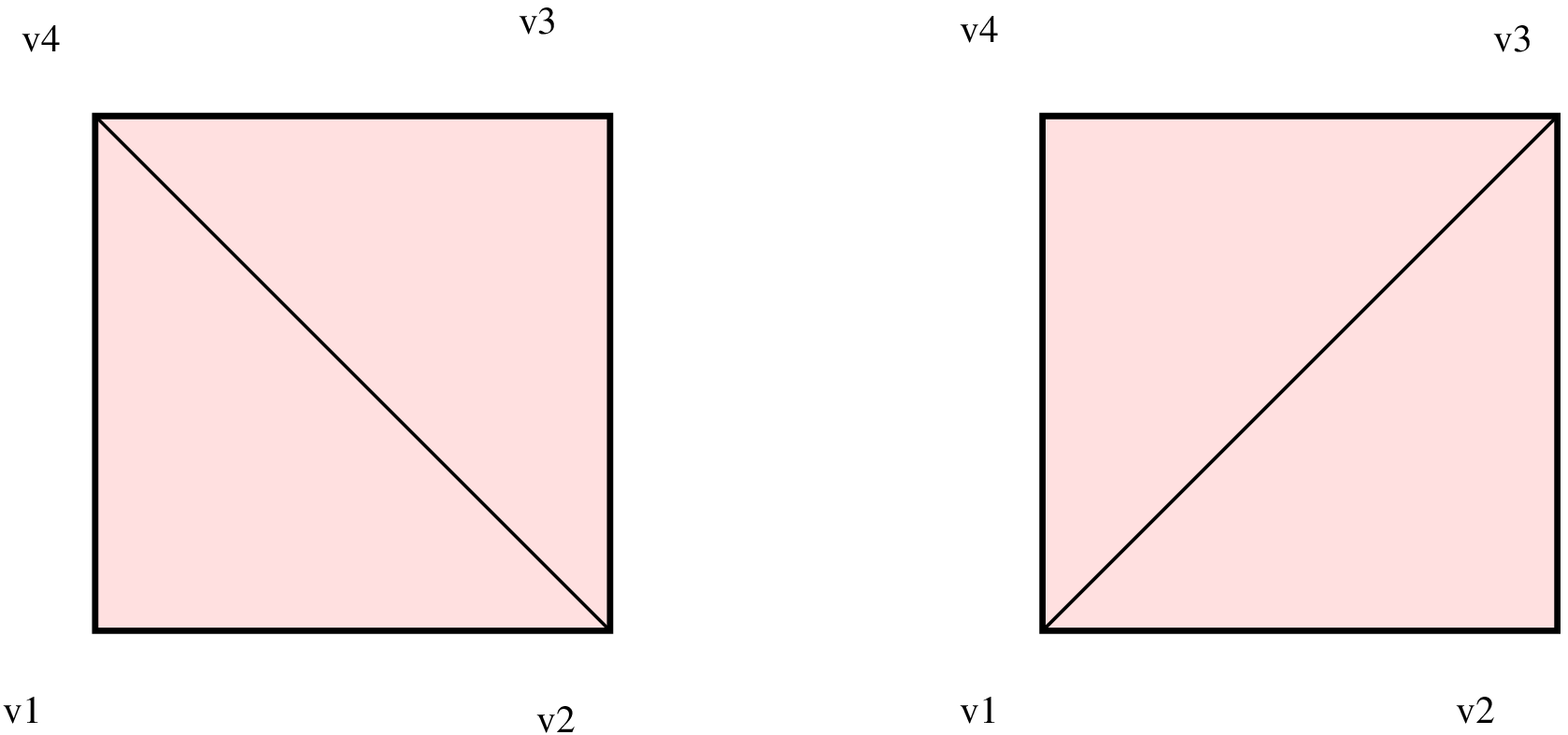}
\caption{Case 2 triangulation}
\label{case2fig}
\end{figure}

The last factor is the difference of two adjoints with respect to the 
triangulations of the quadrilateral in Figure~\ref{case2fig}. 
The vanishing can be seen directly: 
write ${\bf n}_1,\ldots, {\bf n}_4$ for 
${\bf n}_{i-1},{\bf n}_{i+1}, {\bf n}_{j-1},{\bf n}_{j+1}$. Then the last factor is
\[
|\bn_{2} \bn_{3} \bn_{4}|\ell_{1}-
|\bn_{1} \bn_{3} \bn_{4}|\ell_{2}+
|\bn_{1} \bn_{2} \bn_{4}|\ell_{3}-
|\bn_{1} \bn_{2} \bn_{3}|\ell_{4}.
\]
Applying $\frac{d}{dx}$ to this shows the $x$ coefficient is 
\[
|\bn_{2} \bn_{3} \bn_{4}|\bn_{11}-
|\bn_{1} \bn_{3} \bn_{4}|\bn_{21}+
|\bn_{1} \bn_{2} \bn_{4}|\bn_{31}-
|\bn_{1} \bn_{2} \bn_{3}|\bn_{41}.
\]
This is the determinant
of the matrix of the $\bn_{i}$ with a repeat row for the 
$x$ coordinates $\bn_{i1}$, so it vanishes. Reason similarly for
the $y$ and $z$ coefficients. 
\pagebreak

\noindent \txbf{Case 3:}
Assume without loss of generality  $i+1\neq j-1,$ $j+1\neq k-1,$  and
$k+1=i-1$.
In this case two coefficients vanish in Equation~\eqref{PCeqn} and
after evaluating at Wachspress coordinates we obtain,
\begin{eqnarray*}
w_{ijk}\circ w_d &=&|\bn_{i+1} \bn_{j+1}
\bn_{i-1}|\prod_{\sss{m\in\gamma(i+1,j+1,k+1)}}\ell_m - |\bn_{i+1}
\bn_{j+1} \bn_{k-1}|\prod_{\sss{m\in\gamma(i+1,j+1,k)}} \ell_m-\\
&&| \bn_{i+1} \bn_{j-1} \bn_{i-1}|\prod_{\sss{m\in\gamma(i+1,j,k+1)}}\ell_m +
|\bn_{i+1} \bn_{j-1} \bn_{k-1} |\prod_{\sss{m\in\gamma(i+1,j,k)}}\ell_m+\\
&&|\bn_{i-1} \bn_{j+1} \bn_{k-1}|\prod_{\sss{m\in\gamma(i,j+1,k)}}\ell_m\ \  -\
\ |
\bn_{i-1} \bn_{j-1} \bn_{k-1}|\prod_{\sss{m\in\gamma(i,j,k)}}\ell_m\\
&=&A^2\
\bigg(\prod_{\stackrel{m\in\gamma(i,j,k}{\scriptscriptstyle{i+1,j+1,k+1)}}}
\ell_m\bigg)\ \big(\ |\bn_{i+1} \bn_{j+1}
\bn_{i-1}|\ell_{j-1}\ell_{k-1}- \\
&& |\bn_{i+1}
\bn_{j+1} \bn_{k-1}|\ell_{i-1}\ell_{j-1}-| \bn_{i+1} \bn_{j-1}
\bn_{i-1}|\ell_{j+1}\ell_{k-1}+\\
&&|\bn_{i+1} \bn_{j-1} \bn_{k-1} |\ell_{j+1}\ell_{i-1}+|\bn_{i-1} \bn_{j+1}
\bn_{k-1}|\ell_{i+1}\ell_{j-1}-\\
&&|\bn_{i-1} \bn_{j-1} \bn_{k-1}|\ell_{i+1}\ell_{j+1}\big)
\end{eqnarray*}
The last factor is the difference of adjoints with respect to 
the triangulations of the pentagon in Figure~\ref{case3}. 
\begin{figure}[htp]
\centering
\psfrag{v1}{\Large$\bn_{i-1}$}
\psfrag{v2}{\Large$\bn_{j-1}$}
\psfrag{v3}{\Large$\bn_{i+1}$}
\psfrag{v4}{\Large$\bn_{k-1}$}
\psfrag{v5}{\Large$\bn_{j+1}$}
\includegraphics[scale=.5]{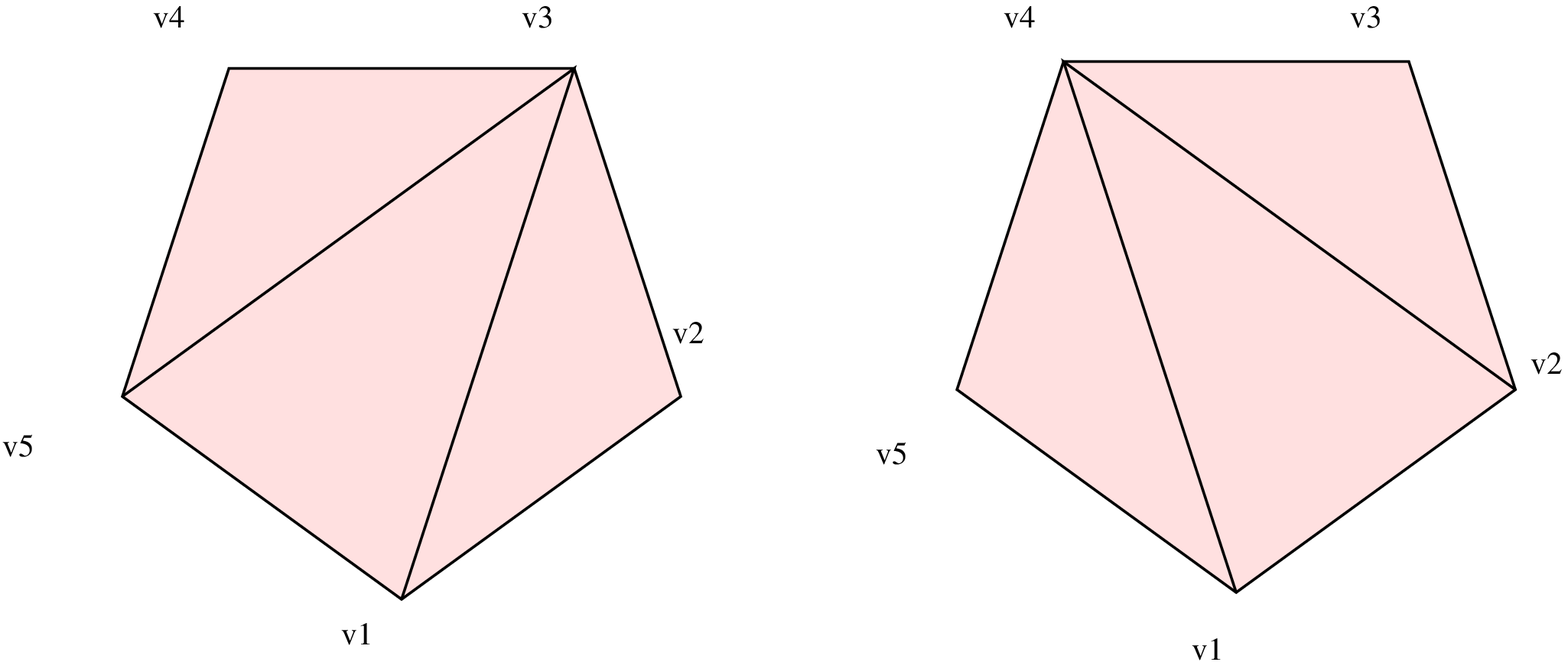}
\caption{Case 3 triangulation}
\label{case3}
\end{figure}

\noindent\txbf{Case 4:} In this case evaluation at Wachspress coordinates yields
{\allowdisplaybreaks\begin{align*}
w_{ijk}\circ w_d&=
|\bn_{i+1} \bn_{j+1} \bn_{k+1}|\prod_{\sss{m\in\gamma(i+1,j+1,k+1)}}\ell_m - 
|\bn_{i+1} \bn_{j+1} \bn_{k-1}|\prod_{m\in\gamma(i+1,j+1,k)}\ell_m-\\
&|\bn_{i+1} \bn_{j-1} \bn_{k+1}|\prod_{m\in\gamma(i+1,j,k+1)}\ell_m + 
|\bn_{i+1} \bn_{j-1} \bn_{k-1}|\prod_{m\in\gamma(i+1,j,k)}\ell_m-\\
& |\bn_{i-1} \bn_{j+1} \bn_{k+1}|\prod_{m\in\gamma(i,j+1,k+1)}\ell_m + 
|\bn_{i-1} \bn_{j+1} \bn_{k-1}|\prod_{m\in\gamma(i,j+1,k)}\ell_m+\\
&|\bn_{i-1} \bn_{j-1} \bn_{k+1}|\prod_{m\in\gamma(i,j,k+1)}\ell_m- 
|\bn_{i-1} \bn_{j-1} \bn_{k-1}|\prod_{m\in\gamma(i,j,k)} \ell_m\\
&= A^2\ \bigg( \prod_{\stackrel{m\in\gamma(i,j,k}{\sss{i+1,j+1,k+1)}}}
\ell_m\bigg)
\big( |\bn_{i+1} \bn_{j+1} \bn_{k+1}|\ell_{i-1}\ell_{j-1}\ell_{k-1} -\\ 
&|\bn_{i+1} \bn_{j+1} \bn_{k-1}|\ell_{i-1}\ell_{j-1}\ell_{k+1}-
|\bn_{i+1} \bn_{j-1} \bn_{k+1}|\ell_{i-1}\ell_{j+1}\ell_{k-1} +\\
&|\bn_{i+1} \bn_{j-1} \bn_{k-1}|\ell_{i-1}\ell_{j+1}\ell_{k+1}-
|\bn_{i-1} \bn_{j+1} \bn_{k+1}|\ell_{i+1}\ell_{j-1}\ell_{k-1} + \\
&|\bn_{i-1} \bn_{j+1} \bn_{k-1}|\ell_{j+1}\ell_{i-1}\ell_{k+1}+
|\bn_{i-1} \bn_{j-1} \bn_{k+1}|\ell_{i+1}\ell_{j+1}\ell_{k-1}- \\
&|\bn_{i-1} \bn_{j-1} \bn_{k-1}| \ell_{i+1}\ell_{j+1}\ell_{k+1}\big)
\end{align*}}
The last factor is the difference of adjoints expressed using the
triangulations of the hexagon in Figure~\ref{case 4}.
\begin{figure}[htp]
 \centering
\psfrag{v1}{\Large$\bn_{i-1}$}
\psfrag{v2}{\Large$\bn_{j-1}$}
\psfrag{v3}{\Large$\bn_{i+1}$}
\psfrag{v4}{\Large$\bn_{k+1}$}
\psfrag{v5}{\Large$\bn_{j+1}$}
\psfrag{v6}{\Large$\bn_{k-1}$}
\includegraphics[scale=.5]{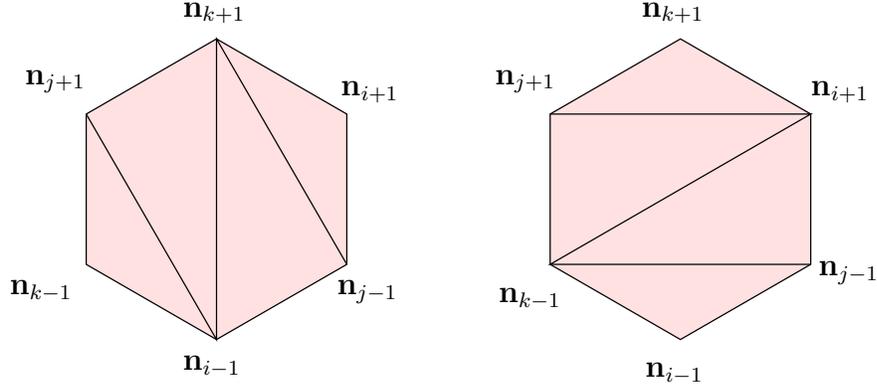}
\caption{Case 4 triangulation}
\label{case 4}
\end{figure}
This completes the analysis when $(i,j,k)$ is a $T$-triple.\\

Next, we consider the situation when $(i,j,k)$ contains a pair of 
consecutive indices. Suppose first that there are exactly two 
consecutive vertices; without loss of generality we assume the indices are 
$(2,3,i)$ with $i>4$.  We have 

\begin{align*}
 w_{2,3,i}:=|\Lambda_2\, \Lambda_3\, \Lambda_i|=&|\bn_2\,
\bn_4\,\bn_{i+1}|x_3x_4x_{i+1}-|\bn_3\, \bn_4\,\bn_{i-1}|x_3x_4x_{i}-\\
&|\bn_3\, \bn_2\,\bn_{i+1}|x_3x_3x_{i+1}+|\bn_3\,
\bn_2\,\bn_{i-1}|x_3x_3x_{i}-\\
&|\bn_1\, \bn_4\,\bn_{i+1}|x_2x_4x_{i+1}+|\bn_1\,
\bn_4\,\bn_{i-1}|x_2x_4x_{i}+\\
&|\bn_1\, \bn_2\,\bn_{i+1}|x_2x_3x_{i+1}-|\bn_1\, \bn_2\,\bn_{i-1}|x_2x_3x_{i}.
\end{align*}
We show $ w_{2,3,i}\circ w_d$ is a multiple of the difference between two
expressions of the adjoint polynomial of a polygon with respect to two
different triangulations. After evaluation at $w_d$ each monomial has a
common factor of $A\prod_{j\neq 2,3}\ell_j$. Thus

\begin{align*}
 \frac{w_{2,3,i}(w_d)}{A\prod_{j\neq 2,3}\ell_j}=&|\bn_2\,
\bn_4\,\bn_{i+1}|\prod_{j\neq 3,4,i+1}\ell_j-|\bn_3\,
\bn_4\,\bn_{i-1}|\prod_{j\neq
3,4,i-1}\ell_j-\\
&|\bn_3\, \bn_2\,\bn_{i+1}|\prod_{j\neq 2,3,i+1}\ell_j+|\bn_3\,
\bn_2\,\bn_{i-1}|\prod_{j\neq 2,3,i-1}\ell_j-\\
&|\bn_1\, \bn_4\,\bn_{i+1}|\prod_{j\neq 1,4,i+1}\ell_j+|\bn_1\,
\bn_4\,\bn_{i-1}|\prod_{j\neq 1,4,i-1}\ell_j+\\
&|\bn_1\, \bn_2\,\bn_{i+1}|\prod_{j\neq 1,2,i+1}\ell_j-|\bn_1\,
\bn_2\,\bn_{i-1}|\prod_{j\neq 1,2,i-1}\ell_j\\
=&\bigg(\prod_{\stackrel{j\in\gamma(2,4}{\sss{i,i+1)}}}
\ell_j\bigg)\cdot\\
&\bigg(|\bn_2\,\bn_4\,\bn_{i+1}|\ell_1\ell_3\ell_{i-1}-|\bn_3\,
\bn_4\,\bn_{i-1}|\ell_1\ell_2\ell_{i+1}-\\
&|\bn_3\, \bn_2\,\bn_{i+1}|\ell_1\ell_4\ell_{i-1}+|\bn_3\,
\bn_2\,\bn_{i-1}|\ell_1\ell_4\ell_{i+1}-\\
&|\bn_1\, \bn_4\,\bn_{i+1}|\ell_2\ell_3\ell_{i-1}+|\bn_1\,
\bn_4\,\bn_{i-1}|\ell_2\ell_3\ell_{i+1}+\\
&|\bn_1\, \bn_2\,\bn_{i+1}|\ell_3\ell_4\ell_{i-1}-|\bn_1\,
\bn_2\,\bn_{i-1}|\ell_3\ell_4\ell_{i+1}\bigg).
\end{align*}
The factor in parentheses is the difference of the adjoints computed with
respect to the triangulations of the polygon in Figure \ref{nonDelta}.
\begin{figure}[htp]
 \centering
\psfrag{v1}{\Large$\bn_{1}$}
\psfrag{v2}{\Large$\bn_{i+1}$}
\psfrag{v3}{\Large$\bn_{3}$}
\psfrag{v4}{\Large$\bn_{4}$}
\psfrag{v5}{\Large$\bn_{i-1}$}
\psfrag{v6}{\Large$\bn_{2}$}
\includegraphics[scale=.5]{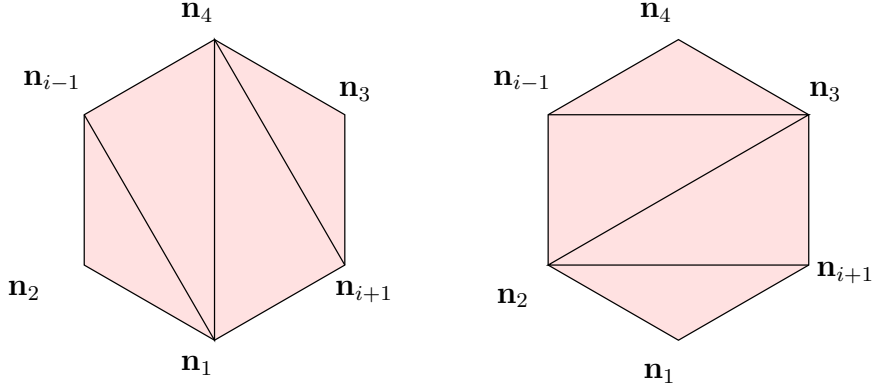}
\caption{Triangulations for the non-$T$-triples}
\label{nonDelta}
\end{figure}

Finally, for the case where the three vertices are consecutive, assume without
loss of generality the triple is $(2,3,4)$, and proceed as above. In 
this case, the triangulations which arise are those which appear in
Figure 5.
\end{proof}
\begin{defn}
$I(d)$ is the ideal generated by the Wachspress quadrics appearing in 
Corollary~\ref{inQuadrics} and the Wachspress cubics appearing in 
Theorem~\ref{cubicgens}.
\end{defn}

\pagebreak
\section{Gr\"obner basis, Stanley-Reisner ring, and free resolution}\label{sec:five}
In this section, we determine the initial ideal of $I(d)$ in graded lex order,
and prove $I(d) = I_{W_d}$. First, some preliminaries.
\subsection{Simplicial complexes and combinatorial commutative algebra}
An abstract $n$-simplex is a set consisting of all subsets of an 
$n+1$ element ground set. Typically a simplex is viewed as a geometric
object; for example a two-simplex on the set $\{a,b,c\}$ can be visualized
as a triangle, with the subset $\{a,b,c\}$ corresponding to the whole 
triangle, $\{a,b\}$ an edge, and $\{a\}$ a vertex. For this reason, 
elements of the ground set are called the vertices. 
\begin{defn}\cite{ziegler}
A simplicial complex $\Delta$ on a vertex set $V$ is a collection of subsets
$\sigma$ of $V$, such that if $\sigma \in \Delta$ and $\tau \subset \sigma$, 
then $\tau \in \Delta$. If $|\sigma| = i+1$ then $\sigma$ is called an $i-$face.
Let $f_i(\Delta)$ denote the number of $i$-faces of $\Delta$, and define 
$\dim(\Delta) = \max\{i \mid f_i(\Delta) \ne 0\}$. If $\dim(\Delta) = n-1$, 
we define $f_\Delta(t) = \sum_{i=0}^n f_{i-1}t^{n-i}$. The ordered list of 
coefficients of $f_\Delta(t)$ is the $f$-vector of $\Delta$, and the 
coefficients of $h_{\Delta}(t):=f_{\Delta}(t-1)$ are the $h$-vector of $\Delta$.
\end{defn}
\begin{exm}\label{1SkelTet}
Consider the one-skeleton of a tetrahedron, with vertices labelled
$\{x_1,x_2,x_3,x_4\}$, as below:
\begin{figure}[htp]
\begin{center}
\psfrag{x4}{$x_4$}
\psfrag{x3}{$x_3$}
\psfrag{x2}{$x_2$}
\psfrag{x1}{$x_1$}
\includegraphics[scale=.15]{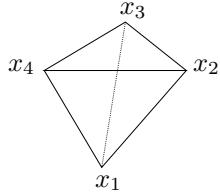}
\caption{One skeleton of a three simplex}
\label{tetra}
\end{center}
\end{figure}
The corresponding simplicial complex $\Delta$ 
consists of all vertices and edges, so $\Delta = \{\emptyset, \{x_i\}, \{x_i, x_j\} \mid 1 \le i \le 4 \mbox{ and } i < j \le 4\}$. Thus, $f(\Delta) = (1,4,6)$ and $h(\Delta) = (1,2,3)$; 
the empty face gives $f_{-1}(\Delta) =1$.\end{exm}
A simplical complex $\Delta$ can be used to define a 
commutative ring, known as the Stanley-Reisner ring.
This construction allows us to use tools of commutative 
algebra to prove results about the topology or combinatorics of $\Delta$.
\begin{defn}
Let $\Delta$ be a simplicial complex on vertices $\{ x_1, \ldots,x_n \}$. 
The Stanley-Reisner ideal $I_\Delta$ is  
\[
I_\Delta = \langle x_{i_1}\cdots x_{i_j} \mid \{x_{i_1},\ldots,x_{i_j}\} \mbox{ is not a face of }\Delta \rangle \subseteq \kk[x_1,\ldots x_n],
\]
and the Stanley-Reisner ring is $\kk[x_1,\ldots x_k]/I_\Delta$.
\end{defn}
In Example~\ref{1SkelTet}, since $\Delta$ has no two-faces,
\[
I_\Delta = \langle x_1x_2x_3, x_1x_2x_4,x_1x_3x_4,x_2x_3x_4 \rangle
= \cap_{1\le i < j \le 4} \langle x_i, x_j \rangle.
\]
\begin{defn}\label{AssPrimes}
A prime ideal $P$ is associated
to a graded $S$-module $N$ if $P$ is the annihilator of some $n \in N$, and $\Ass(N)$ is the set  of all associated primes of $N$. 
\end{defn}
\begin{defn}\label{CMdef}
For a finitely generated graded $S$-module $N$, $\codim(N) = \min\{\codim(P) 
\mid P \in \Ass(N)\}$. The projective dimension $\pdim(N)$ is the length
of a minimal free resolution of $N$; $N$ is Cohen-Macaulay 
if $\codim(N) = \pdim(N)$. $S/I$ is arithmetically Cohen-Macaulay
if it is Cohen-Macaulay as an $S$-module.
\end{defn}
\subsection{Application to Wachpress surfaces}
\begin{defn}
Define $I_\Gamma(d) \subseteq \kk[x_1,\ldots,x_d]$ as
\[
I_\Gamma(d) = \langle x_1x_3,\ldots,x_1x_{d-1}\rangle + K_{2,d-1},
\]
where $K_{2,d-1}$ consists of all squarefree cubic monomials in $x_2,\ldots,x_{d-1}$.
\end{defn}
\begin{thm}\label{GG}
The quotient $S/I_\Gamma(d)$ is arithmetically Cohen-Macaulay, of Castelnuovo-Mumford regularity two, and has Hilbert series
\[
\HS(S/I_\Gamma(d),t) = \frac{1 + (d-3)t + {d-3 \choose 2}t^2}{(1-t)^3}.
\]
\end{thm}
\begin{proof}
The ideal $I_\Gamma(d)$ is the Stanley-Reisner ideal of a one dimensional simplicial complex $\Gamma$ consisting of a complete graph on vertices $\{x_2, \ldots,x_{d-1}\}$, with a single additional edge $\overline{x_1x_2}$ attached. 
All connected graphs are shellable, so since 
shellable implies Cohen-Macaulay (see \cite{ms}), 
$S/I_\Gamma(d)$ is Cohen-Macaulay.
Since $I_\Gamma(d)$ contains no terms involving $x_d$, if $S'=\kk[x_1,\ldots,x_{d-1}]$, then
\[
S/I_\Gamma(d) \simeq S'/I_\Gamma(d) \otimes \kk[x_d]
\]
The Hilbert series of a Stanley-Reisner ring has numerator equal to 
the $h$-vector of the associated simplicial complex (see \cite{sch}), which in this 
case is a graph on $d-1$ vertices with ${d-2 \choose 2}+1$ edges. 
Converting $f(\Gamma) = (1,d-1,{d-2 \choose 2}+1)$ to $h(\Gamma)$ 
yields the Hilbert series of $S'/I_\Gamma(d)$. The Hilbert series of a 
graph has denominator $(1-t)^2$, and tensoring with $\kk[x_d]$ contributes a 
factor of $\frac{1}{1-t}$, yielding the result. 
\end{proof}

\begin{thm}\label{GBthm}
In graded lex order, $\In_{\prec}I(d) = I_\Gamma(d)$.
\end{thm}
\begin{proof}
First, note that 
\[
I_\Gamma(d) \subseteq \In_{\prec}I(d),
\]
which follows from Corollary~\ref{inQuadrics} and Theorem~\ref{cubicgens}, 
combined with the observation that in graded lex order, $\In(|\Lambda_i \Lambda_j \Lambda_k|) = x_ix_jx_k$ if $i<j<k$, as long as $k \ne d$. Since $I(d) \subseteq I_{W_d}$, there is a surjection
\[
S/I(d) \twoheadrightarrow S/I_{W_d},
\]
hence $HP(S/I(d),t) \ge HP(S/I_{W_d},t)$. Since 
\[
HP(S/I(d),t) = HP(S/\In_{\prec}I(d),t)
\]
and 
\[
I_\Gamma(d) \subseteq \In_{\prec}I(d)
\]
we have
\[
HP(S/I_\Gamma(d),t) \ge HP(S/\In_{\prec}I(d),t) = HP(S/I(d),t) \ge  HP(S/I_{W_d},t).
\]
The Hilbert polynomial $HP(S/I_{W_d},t)$ is given by Equation~\eqref{HPisHP}.
The  Hilbert series of $S/I_\Gamma(d)$ is given by Theorem~\ref{GG}, from
which we can extract the Hilbert polynomial:
\[
\HP(S/I_\Gamma(d),t) = {d-3 \choose 2}{t \choose 2} + (d-3){t+1 \choose 2} + {t+2 \choose 2},
\]
and a check shows this agrees with Equation~\eqref{HPisHP}.
Since $I_\Gamma(d) \subseteq \In_{\prec}I(d)$, equality of the Hilbert polynomials implies that in high degree (i.e. up to saturation) 
\[
I_\Gamma(d) = \In_{\prec}I(d) \mbox{ and } I(d) = I_{W_d}.
\]
Consider the short exact sequence
\[
0\longrightarrow \In_{\prec}I(d)/I_\Gamma(d) \longrightarrow S/I_\Gamma(d) \longrightarrow S/\In_{\prec}I(d) \longrightarrow 0.
\]
By Lemma 3.6 of \cite{eis1},
\begin{equation}
\label{AssP}
\Ass(\In_{\prec}I(d)/I_\Gamma(d)) \subseteq \Ass(S/I_\Gamma(d)).
\end{equation}
Since $HP(S/I_\Gamma(d),t) = HP(S/\In_{\prec}I(d),t)$, the module 
$\In_{\prec}I(d)/I_\Gamma(d)$ must vanish in high degree, so is 
supported at $\mathfrak{m}$, which is of codimension $d$. But 
$I_\Gamma(d)$ is a radical ideal supported in codimension $d-3$, 
so it follows from Equation~\eqref{AssP} that $\In_{\prec}I(d)/I_\Gamma(d)$ must vanish. 
\end{proof}
\begin{cor}\label{IisI}
The ideal $I(d)$ is the ideal of the image of 
\[
X_d \longrightarrow \mathbb{P}(H^0(D_{d-2})).
\]
In particular, $I(d) = I_{W_d}$, and $S/I(d)$ is arithmetically Cohen-Macaulay.
\end{cor}
\begin{proof}
By the results of \S 2 and \S 3, $I(d) \subseteq I_{W_d}$, and the proof
of Theorem~\ref{GBthm} showed that they are equal up to saturation.
Hence, $I_{W_d}/I(d)$ is supported at $\mathfrak{m}$. Consider the short exact sequence
\[
0\longrightarrow I_{W_d}/I(d) \longrightarrow S/I(d) \longrightarrow S/(I_{W_d}) \longrightarrow 0.
\]
Since $S/I_\Gamma(d) = S/\In_{\prec}I(d)$ is arithmetically Cohen-Macaulay 
of codimension $d-3$, by uppersemicontinuity \cite{her} so is $S/I(d)$, 
so $I_{W_d}/I(d) =0$.
\end{proof}
\begin{cor}\label{CMreg}
The quotient $S/I_{W_d}$ has regularity two.
\end{cor}
\begin{proof}
Since $S/I(d)$ is Cohen-Macaulay, reducing modulo a linear regular sequence
of length three yields an Artinian ring with the same regularity, which is
equal to the socle degree \cite{eis}. By Theorem~\ref{GG} and Theorem~\ref{GBthm}, this is two, so the regularity of $S/I_{W_d}$ is two.
\end{proof}

\begin{thm}\label{BettiTable}
The nonzero graded betti numbers of the minimal free resolution of $S/I(d)$ are
given by $b_{12}=d-3$, and for $i\ge 3$ by  
\[
b_{i-2,i} = {d-3 \choose i} - (d-3){d-3 \choose i-1}+ {d-3 \choose 2} {d-3 \choose i-2}.
\] 
\end{thm}
\begin{proof}
By Corollary~\ref{CMreg}, there are only two rows in the betti table 
of $S/I(d)$. By Corollary~\ref{noLsyz}, the top
row is empty, save for the quadratic generators at the first step.
Thus, the entire betti diagram may be obtained from the Hilbert
series, which is given in Theorem~\ref{GG}, and the result 
follows.
\end{proof}
\vskip .1in
\noindent We are at work on generalizing the results here to higher dimensions. 
\vskip .1in
\noindent{\bf Acknowledgments} Computations were performed using Macaulay2,
by Grayson and Stillman, available at: {\tt http://www.math.uiuc.edu/Macaulay2/}. A script is available at {\tt http://www.math.uiuc.edu/$\sim$schenck/Wpress.m2/}. 
Our collaboration began during a SIAM conference on applied algebraic geometry,
and we thank the organizers of that conference, Dan Bates and Frank Sottile. 
We thank Frank Sottile and Greg Smith for useful comments, and the
first author also thanks his thesis advisor Frank Sottile for introducing him to
Wachspress varieties and providing guidance during the dissertation process. 
\bibliographystyle{amsalpha}

\end{document}